\documentclass[11pt]{article}
\usepackage[T1]{fontenc}
\usepackage{graphicx,amsmath,amsthm,amsfonts,amssymb,microtype,thmtools,underscore,mathtools,wrapfig,anyfontsize,thm-restate}
\usepackage[shortlabels]{enumitem}
\setlist[itemize]{topsep=0ex,itemsep=0ex,parsep=0.4ex}
\setlist[enumerate]{topsep=0ex,itemsep=0ex,parsep=0.4ex}
\usepackage[dvipsnames,svgnames,table]{xcolor}
\usepackage[unicode=true]{hyperref}
\hypersetup{
	colorlinks=true,
	breaklinks=true,
	linkcolor={black},
	citecolor={black},
	urlcolor={blue!60!black},
	pdftitle={Short Paths in the Planar Graph Product Structure Theorem},
	pdfauthor={Kevin Hendrey, David~R.~Wood}} 
\usepackage[noabbrev,capitalise]{cleveref}
\crefname{lem}{Lemma}{Lemmas}
\crefname{thm}{Theorem}{Theorems}
\crefname{cor}{Corollary}{Corollaries}
\crefname{prop}{Proposition}{Propositions}
\crefname{conj}{Conjecture}{Conjectures}
\crefname{open}{Open Problem}{Open Problems}
\crefname{claim}{Claim}{Claims}
\crefformat{equation}{(#2#1#3)}
\Crefformat{equation}{Equation #2(#1)#3}

\newcommand{\defn}[1]{\textcolor{Maroon}{\emph{#1}}}

\newcommand{\TT}{\mathcal{T}}

\newcommand{\JJ}{\mathcal{J}}
\newcommand{\GG}{\mathcal{G}}
\newcommand{\BB}{\mathcal{B}}
\newcommand{\DD}{\mathcal{D}}

\usepackage[longnamesfirst,numbers,sort&compress]{natbib}
\makeatletter
\def\NAT@spacechar{~}
\makeatother
\usepackage[margin=30mm]{geometry}
\DeclarePairedDelimiter{\floor}{\lfloor}{\rfloor}
\DeclarePairedDelimiter{\ceil}{\lceil}{\rceil}

\renewcommand{\geq}{\geqslant}
\renewcommand{\leq}{\leqslant}
\newcommand{\subsetsim}{\mathrel{\substack{\textstyle\subset\\[-0.3ex]\textstyle\sim\\[-0.4ex]}}}

\newcommand{\StrongProd}{\mathbin{\boxtimes}}

\DeclareMathOperator{\dist}{dist}

\DeclareMathOperator{\bn}{bn}
\DeclareMathOperator{\ltw}{ltw}
\DeclareMathOperator{\stw}{stw}
\DeclareMathOperator{\tw}{tw}

\DeclareMathOperator{\ppw}{ppw}

\DeclareMathOperator{\bw}{bw}
\DeclareMathOperator{\rad}{rad}
\renewcommand{\thefootnote}{\fnsymbol{footnote}}
\theoremstyle{plain}
\newtheorem{thm}{Theorem}
\newtheorem{lem}[thm]{Lemma}
\newtheorem{cor}[thm]{Corollary}
\newtheorem{obs}[thm]{Observation}
\newtheorem{claim}{Claim}[thm]
\theoremstyle{definition}
\newtheorem{conj}[thm]{Conjecture}
\newcommand{\PP}{\mathcal{P}}
\newcommand{\LL}{\mathcal{L}}

\newcommand{\RR}{\mathbb{R}}

\begin{document}
	
    \title{\bf Short Paths in the \\
    Planar Graph Product Structure Theorem}
	
    \author{Kevin Hendrey\footnotemark[2] 
    \qquad David R. Wood\footnotemark[2] }
	
    \maketitle
	
    \footnotetext[2]{School of Mathematics, Monash   University, Melbourne, Australia\\ (\texttt{\{Kevin.Hendrey1,David.Wood\}@monash.edu}). Research supported by the Australian Research Council. }
	
	\begin{abstract}
		The Planar Graph Product Structure Theorem of Dujmovi\'c~\emph{et~al.}~[J.~ACM~'20] says that every planar graph $G$ is contained in $H\boxtimes P\boxtimes K_3$ for some planar graph $H$ with treewidth at most 3 and some path $P$. This result has been the key to solving several old open problems. Several people have asked whether the  Planar Graph Product Structure Theorem can be proved with good upper bounds on the length of $P$. No $o(n)$ upper bound was previously known for $n$-vertex planar graphs.     We answer this question in the affirmative, by proving that for any $\epsilon\in (0,1)$ every $n$-vertex planar graph is contained in $H\boxtimes P\boxtimes K_{O(1/\epsilon)}$, for some planar graph $H$ with treewidth 3 and for some path $P$ of length $O(\frac{1}{\epsilon}n^{(1+\epsilon)/2})$. This bound is almost tight since there is a lower bound of $\Omega(n^{1/2})$ for certain $n$-vertex planar graphs.  In fact, we prove a stronger result with $P$ of length $O(\frac{1}{\epsilon}\,\tw(G)\,n^{\epsilon})$, which is tight up to the $O(\frac{1}{\epsilon}\,n^{\epsilon})$ factor for every $n$-vertex planar graph $G$. Finally, taking $\epsilon=\frac{1}{\log n}$, we show that every $n$-vertex planar graph $G$ is contained in $H\boxtimes P\boxtimes K_{O(\log n)}$ for some planar graph $H$ with treewidth at most 3 and some path $P$ of length $O(\tw(G)\,\log n)$. This result is particularly attractive since the treewidth of the product $H\boxtimes P\boxtimes K_{O(\log n)}$ is within a $O(\log^2n)$ factor of the treewidth of $G$. 
	\end{abstract}
	
	\newpage
	
	\renewcommand*{\thefootnote}{\arabic{footnote}}
	
	\section{Introduction}\label{sec:intro}
	
A central theme of graph structure theory is to describe complicated graphs in terms of simpler graphs. Graph product structure theory achieves this goal, by describing complicated graphs as subgraphs of products of simpler graphs, typically with bounded treewidth. As defined in \cref{sec:TreeDecompositions}, the treewidth of a graph $G$, denoted by \defn{$\tw(G)$}, is the standard measure of how similar $G$ is to a tree.
	
As illustrated in \cref{ProductExample}, the \defn{strong product} $A \StrongProd B$ of graphs  $A$ and $B$ has vertex-set
	\begin{wrapfigure}[10]{r}{70mm}
		\centering
		\includegraphics{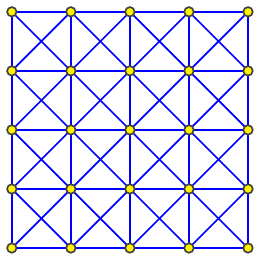}
		\caption{\label{ProductExample} Strong product of two paths.}
	\end{wrapfigure}
	$V(A) \times V(B)$, where distinct vertices $(v,x)$ and $(w,y)$ are adjacent if:
	\begin{itemize}
		\item $v = w$ and $xy \in E(B)$, or 
		\item $x = y$ and $vw \in E(A)$, or
		\item $vw\in E(A)$ and $xy \in E(B)$.
	\end{itemize}

	The following Planar Graph Product Structure Theorem is the classical example of a graph product structure theorem. Here, a graph $H$ is \defn{contained} in a graph $G$ if $H$ is isomorphic to a subgraph of $G$, written \defn{$H \subsetsim G$}. 
	
	\begin{thm}
		\label{PGPST} 
		For every planar graph $G$:
		\begin{enumerate}[(a)]
			\item $G\subsetsim H \StrongProd P$ for some graph $H$ with $\tw(H)\leq 6$ and path $P$ \textup{\citep{UWY22}},
			\item $G\subsetsim H \StrongProd P \StrongProd K_2$ for some graph $H$ with $\tw(H)\leq 4$ and path $P$ \textup{\citep{UWY22}},
			\item $G\subsetsim H \StrongProd P \StrongProd K_3$ for some graph $H$ with $\tw(H)\leq 3$ and path $P$ \textup{\citep{DJMMUW20}}.
		\end{enumerate}
	\end{thm}
	
	\citet{DJMMUW20} first proved \cref{PGPST}(a) with $\tw(H)\leq 8$. In follow-up work, \citet{UWY22} improved the bound to $\tw(H)\leq 6$. Part (b) is due to Dujmovi\'c, and is presented in \citep{UWY22}. Part (c) is in the original paper of  \citet{DJMMUW20}. \citet{ISW24} gave a new proof of part (c). \citet{HJ24} establish analogous results where $G$ is an induced subgraph of $H\boxtimes P\boxtimes K_c$. 
	
	\cref{PGPST} provides a powerful tool for studying questions about planar graphs, by reducing to graphs of bounded treewidth. Indeed, this result has been the key for resolving several open problems regarding queue layouts~\cite{DJMMUW20}, nonrepetitive colourings~\cite{DEJWW20}, centred colourings~\cite{DFMS21}, adjacency labelling schemes~\cite{GJ22,BGP22,EJM23,DEGJMM21}, twin-width~\cite{BDHK24,JP22,KPS24}, infinite graphs~\cite{HMSTW}, comparable box dimension~\cite{DGLTU22}, and transducibility~\citep{GPP,HJ25}. 
	
	In several of these applications, because the dependence on $\tw(H)$ is often exponential, the best bounds are obtained by applying the 3-term product in \cref{PGPST}(c). The $\tw(H)\leq 3$ bound in \cref{PGPST}(c) is best possible in any result saying that every planar graph is contained $H \StrongProd P \StrongProd K_c$ where $P$ is a path (see \citep{DJMMUW20}). 
	
	The following natural question\footnote{This question was independently asked by Pat Morin and Alex Scott, and probably others.} arises in \cref{PGPST}: How short can we make the path $P$? Consider a planar graph $G$ with $n$ vertices. It is easy to see that one can assume that $|V(H)|\leq n$ and $|V(P)|\leq n$. Moreover, the proofs of \cref{PGPST} show that $|V(P)|\leq\rad(G)+1$. However, $\rad(G)$ can be $\Theta(n)$. So no $o(n)$ upper bound on $|V(P)|$ is known. It is necessary that $|V(P)|\geq \Omega(\sqrt{n})$, since there are $n$-vertex planar graphs with treewidth $\Theta(\sqrt{n})$, implying\footnote{It is easily seen that 
		$\tw(H_1 \StrongProd H_2) \leq (\tw(H_1)+1)|V(H_2)|-1$, by replacing each bag $B_x$ of an optimal tree-decomposition of $H_1$ with a bag $B'_x:=\{(v,w)\in V(H_1\boxtimes H_2):v\in B_x\}$.
        }  
	\[\Omega(\sqrt{n}) \leq \tw(G)\leq \tw(H\boxtimes P)\leq (\tw(H)+1)|V(P)|-1 \leq O(|V(P)|).\]
	The first contribution of this paper is to establish the following upper bound on $|V(P)|$ that almost matches this lower bound. 

	\begin{thm}
		\label{PGPST-sqrtn} 
		For any $\epsilon\in (0,\frac12)$ and integer $n\geq 1$, every $n$-vertex planar graph $G$ is contained in $H \StrongProd P \StrongProd K_c$ for some planar graph $H$ with $\tw(H)\leq 3$, some path $P$ with 
		$|V(P)|\leq (\frac{32}{\epsilon}+52) n^{(1+\epsilon)/2}$, and some integer $c\leq 24+\frac{8}{\epsilon}$.
	\end{thm}
	
	The above-mentioned lower bound says that $|V(P)|\geq \Omega(\tw(G))$ in \cref{PGPST}. 
	We prove the following almost-matching upper bound. \cref{PGPST-tw}  implies and strengthens \cref{PGPST-sqrtn} since $\tw(G)\leq 2\sqrt{3n}-1$ for every $n$-vertex planar graph $G$ (see \citep[Lemma~10~and~Theorem~12]{DMW17}).  

	\begin{thm}
		\label{PGPST-tw} 
		For any $\epsilon\in (0,\frac12)$ and integer $n\geq 1$, every $n$-vertex planar graph $G$ is contained in $H \StrongProd P \StrongProd K_c$ for some planar graph $H$ with $\tw(H)\leq 3$, some path $P$ with $|V(P)|\leq (\frac{9}{\epsilon}+15 )(\tw(G)+1)^{1-\epsilon}\,n^{\epsilon}$, and some integer $c\leq 24+\frac{8}{\epsilon}$.
	\end{thm}
	
	We now compare \cref{PGPST-tw} with previous similar results in terms of $\tw(G)$. Improving on earlier work in \citep{UTW}, \citet{ISW24} showed that every planar graph $G$ is contained in $H\boxtimes K_{\tw(G)+1}$ where $\tw(H)\leq 3$. \cref{PGPST-tw} compares favourably to this result, since in \cref{PGPST-tw} there is a path of length $O(\tw(G)n^\epsilon)$ rather than a complete graph of order $\tw(G)+1$.
	
	\cref{PGPST-tw} with $\epsilon=\frac{1}{\log_2 n}$ implies:
	
	\begin{cor}
		\label{PGPST-log} 
		For any integer $n\geq 1$, every $n$-vertex planar graph $G$ is contained in $H \StrongProd P \StrongProd K_c$ for some planar graph $H$ with $\tw(H)\leq 3$, some path $P$ with $|V(P)|\leq (18\log_2n+30) (\tw(G)+1)$, and some integer $c\leq 24 +8\log_2 n$.
	\end{cor}
	
	This result is particularly attractive since 
	$H \StrongProd P \StrongProd K_c$ has treewidth  $O(\tw(G)\,\log^2n)$. \citet{DMWW} asked whether for every planar graph $G$, there exists a bounded treewidth graph $H$ and a path $P$ such that $G\subsetsim H\boxtimes P$ and $\tw(H\boxtimes P) \in O(\tw(G))$? \cref{PGPST-log} answers this question within a $\log^2 n$ factor. 
		
    We now give a high-level sketch of our proof. Our starting point is a tool of \citet{ISW24} that converts a tree-decomposition of a graph $G$ excluding a particular minor into a partition of $G$ (see \cref{TreeDecompPartition,TreeDecompPartition-Kst}, \cref{sec:prodpartlay}). This enables us to work with the seemingly weaker notion of layered tree-decompositions, instead of graph products. 
    \citet{BV13} showed that planar graphs can be triangulated without significantly increasing their treewidth, which allows us to restrict our attention to planar triangulations. 
    Given a planar triangulation $G$, our goal now is to find a tree-decomposition and a layering of $G$, where each bag of the tree-decomposition intersects each layer in a bounded number of vertices, and the number of layers is `small'. 
    The number of layers we use will end up corresponding to the length of the path in the product. 
    In \cref{sec:TreeDecompositions}, we explain how the tree-decomposition of $G$ will be generated from a carefully constructed spanning tree $T$ of $G$.
    We use an extension of a method of \citet{Eppstein99}, where the underlying tree of the decomposition is a spanning tree of the dual $G^*$. Every bag of this decomposition is the union of two paths in $T$, each with the property that their end-vertices are adjacent in $G$. 

    We now need to find a spanning tree $T$ and a layering with few layers, such that for each edge $vw$ of $G$, the $vw$-path in $T$ intersects a bounded number of layers. The standard way to produce a layering is to partition the vertices according to their distance from a root vertex (BFS layering). This approach can produce $\Omega(n)$ layers, which is too big for our purposes. A natural variant partitions the vertices according to their distance from a set of root vertices. This produces a BFS spanning forest $F$ with one root vertex in each component. For this approach to work, we would need each vertex to be `close' to some root vertex, so that the number of layers is small. 
    For any integer $r\geq 1$, the following additional properties would then suffice: (a) the root vertices are at pairwise distance greater than $2r+1$, and (b) each component of the graph obtained from $G$ by deleting the radius-$r$ balls centred at the root vertices is adjacent to at most two of these balls.
    With these properties, one can use the fact that minimal separators in planar triangulations are connected (see \cref{TriangulationSeparator}, \cref{sec:tree-partitions}) together with the observation that each component of $F$ contains exactly one of these radius-$r$ balls to deduce that contracting the edges of $F$ will reduce $G$ to a tree.
    Then for each edge $vw$ of $G$, we see that the $vw$-path in $T$ intersects at most two components of $F$ and thus intersects each layer in a bounded number of vertices, as desired. 
    
    Our proof in fact requires a more complicated variant of this approach, where the layering is computed iteratively. At each step, the next layer consists of the neighbours of the previous layer, plus a carefully chosen set of new roots (not already assigned to previous layers). 
    Using this delayed root-introduction strategy, we find a layering and a spanning tree $T$ such that for each edge $vw$, the $vw$-path in $T$ intersects a bounded number of components, which suffices for our purposes.
    This is the core of the proof, and is presented in \cref{sec:FindingSeeds}.
    An important step towards finding this layering is to show that $G$ has a tree-partition in which each part has weak-diameter $O(\tw(G))$ (see \cref{lem:twdiamparts}, \cref{sec:tree-partitions}), a result of independent interest. 

Our final contributions show that various beyond planar graph classes have product structure theorems akin to \cref{PGPST-sqrtn}. This includes fan-planar graphs, $k$-planar graphs, and powers of planar graphs with bounded degree. These results are presented in \cref{sec:Extensions}.
    
\section{Tree-decompositions}
\label{sec:TreeDecompositions}
	
	We consider finite simple undirected graphs $G$ with vertex-set $V(G)$ and edge-set $E(G)$. For a vertex $v\in V(G)$, let $N_G(v):=\{w\in V(G): vw\in E(G)\}$.
	To \defn{contract} an edge $vw$ in a graph $G$ means to delete $v$ and $w$, and add a new vertex with neighbourhood $(N_G(v)\cup N_G(w))\setminus\{v,w\}$. 
		A graph $H$ is a \defn{minor} of a graph $G$ if a graph isomorphic to $H$ can be obtained from a subgraph of $G$ by contracting edges. A graph $G$ is \defn{$H$-minor-free} if $H$ is not a minor of $G$. 
		
		Given a tree $T$ and vertices $v,w\in V(T)$, we denote by $T_{v,w}$ the unique path from $v$ to $w$ in $T$.
		
		For a tree $T$ with $V(T)\neq\emptyset$, a \defn{$T$-decomposition} of a graph $G$ is a collection $(B_x:x \in V(T))$ such that:
		\begin{itemize}
			\item $B_x\subseteq V(G)$ for each $x\in V(T)$, 
			\item for every edge ${vw \in E(G)}$, there exists a node ${x \in V(T)}$ with ${v,w \in B_x}$, and 
			\item for every vertex ${v \in V(G)}$, the set $\{ x \in V(T) : v \in B_x \}$ induces a non-empty (connected) subtree of $T$. 
		\end{itemize}
		A \defn{tree-decomposition} is a $T$-decomposition for any tree $T$. 
		The \defn{width} of a $T$-decomposition $(B_x:x \in V(T))$ is ${\max\{ |B_x| : x \in V(T) \}-1}$. The \defn{treewidth} of a graph $G$, denoted \defn{$\tw(G)$}, is the minimum width of a tree-decomposition of $G$. 
		Treewidth is the standard measure of how similar a graph is to a tree. Indeed, a connected graph has treewidth at most 1 if and only if it is a tree. It is an important parameter in structural graph theory, especially Robertson and Seymour's graph minor theory, and also in algorithmic graph theory, since many NP-complete problems are solvable in linear time on graphs with bounded treewidth. See \citep{HW17,Bodlaender98,Reed97} for surveys on treewidth. 
		
		Two subgraphs $A$ and $B$ of a graph $G$ \defn{touch} if $V(A)\cap V(B)\neq\emptyset$, or some edge of $G$ has one endpoint in $A$ and the other endpoint in $B$. A \defn{bramble} in $G$ is a set of connected subgraphs of $G$ that pairwise touch. A set $S$ of vertices in $G$ is a \defn{hitting set} of a bramble $\BB$ if $S$ intersects every element of $\BB$. The \defn{order} of $\BB$ is the minimum size of a hitting set. The \defn{bramble number} of $G$ is the maximum order of a bramble in $G$. Brambles were first defined by \citet{ST93}, where they were called \defn{screens of thickness $k$}. \citet{ST93} proved the following result. 
		
		\begin{thm}[Treewidth Duality Theorem, \citep{ST93}]
			\label{TreewidthDuality}
			For every graph $G$, 
			$$\tw(G)+1=\bn(G).$$ 
		\end{thm}
		
		A tree-decomposition $(B_x:x\in V(T))$ of a graph $G$ is \defn{$k$-simple}, for some integer $k\geq 0$,  if it has  width  at most $k$, and for every set $S$ of $k$ vertices in $G$, we have $|\{x\in V(T): S\subseteq B_x\}|\leq 2$. The \defn{simple treewidth} of a graph $G$, denoted by $\stw(G)$, is the minimum integer $k\geq 0$ such that $G$ has a $k$-simple tree-decomposition. Simple treewidth appears in several places in the literature under various guises \citep{KU12,KV12,MJP06,Wulf16}. The following facts are well known: A connected graph has simple treewidth 1 if and only if it is a path. A graph has simple treewidth at most 2 if and only if it is outerplanar. A graph has simple treewidth at most 3 if and only if it has treewidth 3 and is planar~\citep{KV12}. The edge-maximal  graphs with simple treewidth 3 are ubiquitous objects, called  \defn{planar 3-trees} in structural graph theory and graph drawing~\citep{AP-SJADM96,KV12}, called \defn{stacked polytopes} in polytope theory~\citep{Chen16}, and called \defn{Apollonian networks} in enumerative and random graph theory~\citep{FT14}. It is also known and easily proved that $\tw(G) \leq \stw(G)\leq \tw(G)+1$ for every graph $G$ (see \citep{KU12,Wulf16}). 
		
		To prove \cref{PGPST-tw} it suffices to work with planar triangulations by the following lemma of \citet{BV13}.
		
		\begin{lem}[{\protect\citep[Lemma~1]{BV13}}]
			\label{Triangulate}
			For every planar graph $G$ with $|V(G)|\geq 3$ there is a planar triangulation $G'$ with $V(G')=V(G)$ and $E(G)\subseteq E(G')$ and $\tw(G')=\max\{\tw(G),3\}$.
		\end{lem}
		
		The next lemma generalises a construction of a tree-decomposition due to \citet{Eppstein99}. For a plane graph $G$, let \defn{$G^*$} be the plane dual of $G$.
		
		\begin{lem}
			\label{GenerateFromSpanningTree}    
			Let $G$ be a plane triangulation, let $T$ be a spanning tree of $G$, and let $T^*$ be the spanning tree of $G^*$ such that two faces of $G$ are adjacent in $T^*$ if and only if there is an edge in $E(G)\setminus E(T)$ that is incident to both faces. For each face $f$ of $G$, let $B_f:=V(T_{v,w})\cup V(T_{w,z})$, where $v,w$ and $z$ are the three vertices incident to $f$. Then $(B_f:f\in V(T^*))$ is a $T^*$-decomposition of $G$.
		\end{lem}
		
		\begin{proof}
			\citet{vonStaudt} showed that $T^*$ is a spanning tree of $G^*$.  
			First note that every edge $e\in E(G)$ is incident to some face $f$, and $e\subseteq B_f$. Now consider a vertex $v$ of $G$, and the set $S_v:=\{f\in V(T^*):v\in B_f\}$.
			Note for any three vertices $v',w',z'\in V(T)$, we have $T_{v',z'}\subseteq T_{v',w'}\cup T_{w',z'}$.
			Thus a face $f\in V(T^*)$ is in $S_v$ if and only if there is an edge $uw\in E(G)$ incident to $f$ such that $v\in V(T_{u,v})$.
			That is, either $v\in \{u,w\}$ or $u$ and $w$ are in different components of $T-v$.
			Let $t=\deg_G(v)$, let $\{v_i:i\in \mathbb{Z}_t\}$ be the neighbours of $v$ in clockwise order, and for each $i\in \mathbb{Z}_t$ let $C_i$ be the component of $T-v$ containing $v_i$, and let $f_i$ be the face of $G$ incident to $v$, $v_i$ and $v_{i+1}$ (we may assume this face is unique since otherwise $G\cong K_3$ and the result is trivial.
			Note that for each $i\in \mathbb{Z}_t$, the edge cut $E(V(G)\setminus V(C_i),V(C_i))$ corresponds to a cycle $O_i$ in $G^*$. Furthermore, $f_{i-1}$ and $f_i$ are adjacent in $O_i$, and the edge connecting them is the unique edge in $E(O_i)\setminus E(T^*)$. Thus, $P_i:=T^*[V(O_i)]$ is a path in $T^*$.
			With this definition, note that $P_i$ and $P_{i+1}$ intersect in the vertex $f_i$, and so 
			$\bigcup \{P_i:i\in \mathbb{Z}_t\}$ is connected.
			Observe that $V(\bigcup \{P_i:i\in \mathbb{Z}_t\})$ contains all and only the faces of $G$ which are incident to some edge whose end-vertices lie in distinct components of $T-v$. 
			Thus, $S_v$ induces a connected subgraph of $T^*$, as required.\end{proof}
		
		We say the tree-decomposition in \cref{GenerateFromSpanningTree} is \defn{generated by $T$}.
		
		\section{Products, partitions and layerings}
        \label{sec:prodpartlay}
		
		The following definitions  provide a useful way to think about graph products. 
		For graphs $H$ and $G$, an \defn{$H$-partition} of $G$ is a collection $(B_x:x\in V(H))$ such that:
		\begin{itemize}
			\item $\bigcup_{x\in V(H)} B_x=V(G)$, 
			\item $B_x\cap B_y=\emptyset$ for distinct $x,y\in V(H)$, and
			\item for each edge $vw\in E(G)$, if $v\in B_x$ and $w\in B_y$, then $x=y$ or $xy\in E(H)$. 
		\end{itemize}
		The \defn{width} of an $H$-partition $(B_x:x\in V(H))$ is~${\max\{ |{B_x}| \colon x \in V(H)\}}$. 
		
		A $P$-partition, where $P$ is a path, is called a \defn{path-partition} or \defn{layering}, which we describe by a sequence $\LL=(L_0,L_1,\dots)$ of subsets of $V(G)$ such that for every edge $vw\in E(G)$, if $v\in L_i$ and $w\in L_j$ then $|i-j|\leq 1$. Each set $L_i$ is called a \defn{layer}. We write $|\mathcal{L}|$ for the number of non-empty layers of $\mathcal{L}$. For example, if $r$ is a vertex in a connected graph $G$ and $L_i:=\{v\in V(G):\dist_G(r,v)=i\}$ for each integer $i\geq 0$, then $(L_0,L_1,\dots)$ is a layering of $G$, called  a \defn{BFS-layering} of $G$.
		The \defn{path-partition-width} of a graph $G$, denoted by \defn{$\ppw(G)$}, is the minimum width of a path-partition of $G$. Note that path-partition-width of a graph $G$ is closely tied to the \defn{bandwidth} of $G$, which is the minimum integer $k$ such that there is a bijection $f:V(G)\rightarrow\{1,\dots,|V(G)|\}$ with $|f(v)-f(w)|\leq k$ for each $vw\in E(G)$. In particular, $\ppw(G)\leq\bw(G) \leq 2\ppw(G)-1$ (see \citep{DJMMW25,BPTW10}).
		
		The next observation characterises when a graph is contained in $H \StrongProd P \StrongProd K_c$. 
		
		\begin{obs}[\citep{DJMMUW20}]
			\label{PartitionProduct}
			For any graph $H$, path $P$, and  integer $c\geq 1$, a graph $G$ is contained $H\StrongProd P \StrongProd K_c$ if and only if $G$ has an $H$-partition $(B_x:x\in V(H))$ and a layering $\LL$ such that $|\LL|\leq|V(P)|$, 
			and $|B_x\cap L|\leq c$ for each $x\in V(H)$ and $L\in\LL$. 
		\end{obs}
		
		The \defn{layered treewidth} of a graph $G$, denoted by $\ltw(G)$, is the minimum integer $k\geq 0$ such that $G$ has a layering $\mathcal{L}$ and a tree-decomposition $(B_x:x\in V(T))$ such that $|L_i\cap B_x|\leq k$ for each $L_i\in \mathcal{L}$ and each $x\in V(T)$. This definition was independently introduced by \citet{Shahrokhi13} and \citet{DMW17}. The later authors proved that $\ltw(G)\leq 3$ for every planar graph $G$ (using a variant of \cref{GenerateFromSpanningTree}). Their proof gives no non-trivial bound on $|\mathcal{L}|$. The next lemma, which is proved at the end of \cref{sec:FindingSeeds} and is the main technical contribution of this paper,  gives such a bound. 
		
		\begin{restatable}{lem}{MainLemName}
			\label{MainLemma} 
			For any $\epsilon\in (0,\frac12)$ and integer $n\geq 1$, every planar graph $G$ with $n$ vertices admits a layering $\mathcal{L}$ and tree-decomposition $(B_x:x\in V(T))$ such that $|\mathcal{L}|\leq (\frac{9}{\epsilon}+15) (\tw(G)+1)^{1-\epsilon}\,n^{\epsilon}$ and $|B_x\cap L|\leq 12+\frac{4}{\epsilon}$ for each $x\in V(T)$ and $L\in\LL$.
		\end{restatable}

		We now show how \cref{MainLemma} and results of \citet{ISW24} imply our main result, \cref{PGPST-tw}. 
		
		\citet{ISW24} introduced the following definition, which measures the `complexity' of a set of vertices with respect to a tree-decomposition $\DD=(B_x:x\in V(T))$ of a graph $G$. Define the \defn{$\DD$-width} of a set $S\subseteq V(G)$ to be the minimum number of bags of $\DD$ whose union contains $S$. 
		The \defn{$\DD$-width} of an $H$-partition $(B_x:x\in V(H))$ of a graph $G$ is the maximum $\DD$-width of one of the parts $B_x$. \citet[Theorem~12]{ISW24}  proved the following result, which converts a tree-decomposition of a graph $G$ (excluding certain minors) into an $H$-partition where $H$ has bounded treewidth. 
		Here \defn{$\JJ_{s, t}$} is the class of graphs $J$ whose vertex-set has a partition $A \cup B$, where $|A| = s$ and $|B| = t$, $A$ is a clique, every vertex in $A$ is adjacent to every vertex in $B$, and $J[B]$ is connected. 
		A graph $G$ is \defn{$\JJ_{s,t}$-minor-free} if no graph in $\JJ_{s,t}$ is a minor of $G$.
		
		\begin{thm}[\citep{ISW24}]
			\label{TreeDecompPartition}
			For any integers $s,t \geq 2$, for any $\JJ_{s,t}$-minor-free graph $G$, for any tree-decomposition $\DD$ of $G$, $G$ has an $H$-partition of $\DD$-width at most $t-1$, where $\tw(H)\leq s$.
		\end{thm}
		
		\citet{ISW24} noted that \cref{TreeDecompPartition,PartitionProduct} imply that every $\JJ_{s,t}$-minor-free graph $G$ with $\ltw(G)\leq c$ is contained in $H\boxtimes P\boxtimes K_{(t-1)c}$, for some graph $H$ with $\tw(H)\leq s$ and for some path $P$. In fact, \cref{TreeDecompPartition,PartitionProduct} imply the following slightly stronger statement. 
		
		\begin{lem}
			\label{LayeringTreeDecompProduct}
			For any integers $s, t \geq 2$, for any $\JJ_{s,t}$-minor-free graph $G$, if $G$ has a layering $\LL$ and a tree-decomposition $(B_x:x\in V(T))$ with $|B_x\cap L|\leq c$ for each $x\in V(T)$ and $L\in\LL$, then $G$ is contained in $H\boxtimes P\boxtimes K_{(t-1)c}$, for some graph $H$ with $\tw(H)\leq s$ and for some path $P$ with $|V(P)|\leq|\LL|$. 
		\end{lem}
		
		Since $\JJ_{s,2}=\{K_{s+2}\}$ and planar graphs are $K_5$-minor-free, \cref{LayeringTreeDecompProduct} is applicable for planar graphs with $s=3$ and $t=2$. With \cref{MainLemma} this implies:
		
		
		
		\begin{thm}
			\label{PGPST-tw-nonplanar} 
			For any $\epsilon\in (0,\frac12)$ and  integer $n\geq 1$, every planar graph $G$ with $n$ vertices is contained in $H \StrongProd P \StrongProd K_c$ for some graph $H$ with $\tw(H)\leq 3$, some path $P$ with $|V(P)|\leq 
			\frac{33}{2\epsilon} n^\epsilon (\tw(G)+1)^{1-\epsilon}$, and some integer $c\leq 12 +\frac{4}{\epsilon}$.
		\end{thm}
		
		\cref{PGPST-tw-nonplanar} improves the bound on $c$ in \cref{PGPST-tw} by a factor of 2, at the expense of $H$ not necessarily being planar. To prove 
		\cref{PGPST-tw} with $H$ planar, 
		we use the following analogue of \cref{TreeDecompPartition} for simple treewidth, established in the arXiv version of the paper by \citet[Theorem~27]{ISW-arXiv}. Here \defn{$K^{\ast}_{s,t}$} is the graph whose vertex-set can be partitioned $A \cup B$, where $|A| = s$, $|B| = t$,  $A$ is a clique, $B$ is an independent set, and every vertex in $A$ is adjacent to every vertex in $B$; that is, $K^{\ast}_{s, t}$ is obtained from $K_{s, t}$ by adding all the edges inside the part of size $s$.
		
		\begin{thm}[\citep{ISW-arXiv}]
			\label{TreeDecompPartition-Kst}
			For any integers $s, t \geq 2$, for any $K^*_{s,t}$-minor-free graph $G$, for any tree-decomposition $\DD$ of $G$, $G$ has an $H$-partition of $\DD$-width at most $t-1$, where $\stw(H)\leq s$.
		\end{thm}

		\cref{TreeDecompPartition-Kst,PartitionProduct} imply the following variant of \cref{LayeringTreeDecompProduct}. 
		
		\begin{lem}
			\label{SimpleLayeringTreeDecompProduct}
			For any integers $s, t \geq 2$, for any $K^{\ast}_{s,t}$-minor-free graph $G$, if $G$ has a layering $\LL$ and a tree-decomposition $(B_x:x\in V(T))$ with $|B_x\cap L|\leq c$ for each $x\in V(T)$ and $L\in\LL$, then $G$ is contained in $H\boxtimes P\boxtimes K_{(t-1)c}$, for some graph $H$ with $\stw(H)\leq s$ and for some path $P$ with $|V(P)|\leq|\LL|$. 
		\end{lem}
		
		Since planar graphs are $K^{\ast}_{3,3}$-minor-free, \cref{SimpleLayeringTreeDecompProduct} is applicable for planar graphs with $s=t=3$. With \cref{MainLemma}, this implies \cref{PGPST-tw}, since (as mentioned above) a graph $H$ has $\stw(H)\leq 3$ if and only if $H$ is planar and $\tw(H)\leq 3$. 
		

		
		
		\section{Tree-partitions}
		\label{sec:tree-partitions}
		
		A \defn{tree-partition} is a $T$-partition where $T$ is a tree. 
		Tree-partitions were independently introduced by \citet{Seese85} and \citet{Halin91}, and have since been widely investigated \citep{Bodlaender-DMTCS99,BodEng-JAlg97,DO95,DO96,Edenbrandt86,Wood06,DKCPS,Wood09,BGJ22,DS20,DW24}. The primary focus of this work has been the width of tree-partitions. For our purposes, the following parameter is important. The \defn{weak-diameter} of a tree-partition $(B_x:x\in V(T))$ of a graph $G$ is the maximum of $\dist_G(v,w)$ taken over all vertices $v$ and $w$ in the same part $B_x$. 
		A tree-partition $(B_x:x\in V(T))$ is \defn{rooted} at $r\in V(T)$ if $T$ is rooted at $r$. A rooted tree-partition $(B_x:x\in V(T))$ is \defn{parent-dominated} if for each edge $vw\in E(T)$, if $w$ is the parent of $v$, then every vertex in $B_v$ has at least one neighbour in $B_w$.

		We need the following folklore lemma, whose proof we include for completeness. 
		
		\begin{lem}
			\label{TriangulationSeparator}
			For any connected subgraphs $H_1$ and $H_1$ of a planar triangulation $G$, every minimal set $S\subseteq V(G)$ separating $V(H_1)$ and $V(H_2)$ induces a connected subgraph of $G$.
		\end{lem}
		
		\begin{proof}
			If $V(H_1)\subseteq S$ then $S=V(H_1)$ by the minimality of $S$, implying $G[S]$ is connected, as desired. 
			Now assume that $V(H_1)\setminus S\neq\emptyset$.
			Let $J_1,\dots,J_k$ be the components of $G-S$ intersecting $H_1$. Let $J:=J_1\cup\dots\cup J_k$. 
			Since $S$ separates $H_1$ and $H_2$, the subgraphs $J$ and $H_2$ are disjoint. Since $H_2$ is connected, $H_2$ lies within a single face $F$ of the embedding of $J$ induced by the embedding of $G$. 
			Let $B_i$ be the set of vertices of $J_i$ on the boundary of $F$. Let $N_i$ be the set of vertices in $V(G)\setminus V(J)$ that are adjacent to at least one vertex in $B_i$. Since $J_i$ is a component of $G-S$, we have $N_i\subseteq S$. One can draw a closed curve $X_i$ in $F$ touching $G$ precisely at the vertices in $N_i$. Since $G$ is a triangulation, consecutive vertices on $X_i$ are adjacent in $G$. So $G[N_i]$ is connected. By construction, $N_i$ separates $J_i$ and $H_2$. 
			If $H_1$ and $S$ are disjoint, then $k=1$ and $N_1$ separates $H_1$ and $H_2$, implying $S=N_1$ by the minimality of $S$, and $G[S]=G[N_1]$ is connected, as desired. Now assume that $H_1$ intersects $S$. 
			Each $J_i$ is connected and intersects $H_1$, which is connected. 
			So $H_1\cup J_1\cup\dots\cup J_k$ is connected. 
			Since  $H_1$ intersects $S$ and $H_1$ is connected, each set $B_i$ has a vertex in $H_1$ which is adjacent to a vertex in $N_i\cap V(H_1)$. Since $G[N_i]$ is connected, $X:=H_1\cup J_1\cup\dots\cup J_k\cup G[N_1]\cup \dots\cup G[N_k]$ is connected.
			Note that all neighbours of $J_i$ in $X-J_i$ are in $N_i$, that $X[N_i]$ is connected and that $J_i$ is disjoint from $N_1\cup\dots\cup N_k$.
			It follows that $X-(J_1\cup\dots\cup J_k)$ is connected.
			By construction, $X-(J_1\cup\dots\cup J_k)$ is the subgraph of $G$ induced by $T:=(N_1\cup \dots\cup N_k)\cup(V(H_1)\cap S)$, which is a subset of $S$. 
			By construction, $T$ separates $H_1$ and $H_2$. By the minimality of $S$, we have $S=T$. Thus $G[S]=G[T]$ is connected.
		\end{proof}
		
		\begin{lem}\label{lem:twdiamparts}
			For every vertex $r$ of a planar triangulation $G$, $G$ has a parent-dominated tree-partition with width at most $3\tw(G)+2$ whose root part is $\{r\}$.
		\end{lem}
		
		\begin{proof}
			Let $k:=\tw(G)+1$. Let $V_i:=\{v\in V(G):\dist_G(r,v)=i\}$ for $i\geq 0$. So ($V_0,V_1,\dots)$ is a layering of $G$. Define a graph $T$ and $T$-partition $(B_x:x\in V(T))$ as follows. For each $i\geq 0$ and each component $C$ of $G[V_i]$, add one vertex $x$ to $T$ with $B_x:=V(C)$. For distinct vertices $x,y\in V(T)$, add the edge $xy$ to $T$ if and only if there is an edge of $G$ between $B_x$ and $B_y$. So $(B_x:x\in V(T))$ is a $T$-partition of $G$. 
			
			Consider a node $x\in V(T)$ with $B_x\subseteq V_i$ with $i\geq 1$. Let $S$ be the set of vertices in $V_{i-1}$ adjacent to at least one vertex in $B_x$. By construction, $G[B_x]$ is connected. So there is a component $C$ of $G[V_i\cup V_{i+1}\cup \dots]$ containing $B_x$. Every vertex that is not in $C$ and is adjacent to a vertex in $C$ is in $S$. So $S$ separates $C$ from $\{r\}$. Moreover, for each vertex $v\in S$, there is a neighbour $w$ of $v$ in $B_x$, and there is a path from $v$ to $r$ internally disjoint from $S$. Thus $S$ is a minimal set separating $V(C)$ from $\{r\}$. By \cref{TriangulationSeparator}, $G[S]$ is connected. Since $S\subseteq V_{i-1}$ there is a node $y$ in $T$ with $S\subseteq B_y$. By the construction of $T$, $x$ has exactly one neighbour $y$ in $T$ with $B_y\subseteq V_{i-1}$. Thus $T$ is a tree. Consider $T$ to be rooted at the node corresponding to part $\{r\}$ (which exists, by the $i=0$ case). For $i\geq 1$, since every vertex in $V_i$ has a neighbour in $V_{i-1}$, the $T$-partition $(B_x:x\in V(T))$ is parent-dominated.
			
			Suppose that $\dist_G(v,w)\geq 3k$ for some $v,w$ in one part $B_{x_i}$ with $B_{x_i} \subseteq V_i$. By the triangle-inequality, $\dist_G(v,w) \leq \dist_G(r,v)+\dist_G(r,w)=2i$, implying $2i\geq 3k$. Let 
			$x_0,x_1,\dots,x_i\in V(T)$ where $x_{i-1}$ is the parent of $x_i$ in $T$. So $B_{x_j}\subseteq V_j$ for each $j\in\{0,\dots,i\}$. Abbreviate $B_{x_j}$ by $B_j$. 
			
			We claim that there are $k+1$ pairwise disjoint paths in $G$ between $B_i$ and $B_{i-k}$. If not, by Menger's Theorem, there is a minimal set $S\subseteq V(G)$ separating $B_i$ and $B_{i-k}$ with $|S|\leq k$. By \cref{TriangulationSeparator}, $G[S]$ is connected. Now 
			\begin{align*}
				3k\leq \dist_G(v,w)
				& \leq \dist_G(v,S)+\dist_G(w,S)+|S|-1 \\
				& \leq
				\dist_G(v,B_{i-k})+\dist_G(w,B_{i-k})+|S|-1 \\
				& \leq 3k-1,
			\end{align*}
			which is a contradiction. Hence there are pairwise disjoint paths $Q_1,\dots,Q_{k+1}$ between $B_i$ and $B_{i-k}$. 
			Let $\beta:= \{Q_a\cup B_b: a\in\{1,\dots,k+1\}, b\in\{i-k,\dots,i\}\}$. 
			By construction, each of $G[B_{i-k}],\dots,G[B_i]$ is connected, and each such path $Q_a$ must intersect each of $B_{i-k},\dots,B_i$. Thus $\beta$ is a bramble. 
			If $S\subseteq V(G)$ and $|S|\leq k$, then $S\cap (Q_a\cup P_b)=\emptyset$ for some $a\in\{1,\dots,k+1\}$ and $b\in\{i-k,\dots,i\}$, implying $S$ is not a hitting set for $\beta$. Hence the order of $\beta$ is at least $k+1$, and $\bn(G)\geq k+1$. By \cref{TreewidthDuality}, $\tw(G)\geq k$, which is a contradiction.
			
			Hence
			$\dist_G(v,w)\leq 3k-1=3(\tw(G)+1)-1=3\tw(G)+2$ for all $v,w$ in any one part $B_x$. Therefore, the $T$-partition $(B_x:x\in V(T))$ has weak-diameter at most $3\tw(G)+2$. 
		\end{proof}
			

		\section{Seeds}
		\label{sec:FindingSeeds}
		
		A \defn{seed} of a connected graph $G$ is an ordered pair $(S,f)$, where $S$ is a nonempty subset of $V(G)$ and $f$ is a function that assigns a non-negative integer to each vertex in $S$ such that for every pair of vertices $v,w\in S$, $|f(v)-f(w)|\geq \dist_G(v,w)$. 
		
		Given a seed $(S,f)$ of a connected graph $G$,
		we define $\mathcal{L}_{G,(S,f)}$ to be the layering $(L_0,L_1,\dots)$ where $L_0:=\{s\in S:f(s)=0\}$ and $L_i:=\left(\{s\in S:f(s)=i\}\cup \bigcup \{N(v):v\in L_{i-1}\}\right)\setminus \bigcup_{j\in [i]}L_{j-1}$ for every integer $i\geq 1$. 
		For each $v\in V(G)$, we denote by $\ell_{G,(S,f)}$ the integer $i$ such that $v\in L_i$.
		Note that $\ell_{G,(S,f)}=\min\{\dist(v,s)+f(s):s\in S\}$.
		To see that this is a layering, note that if a vertex $v$ has a neighbour $w$ such that $\ell_{G,(S,f)}(w)=i$, then there is some $s\in S$ such that $i=f(s)+\dist(w,s)$, which means $f(s)+\dist(v,s)\leq i+1$ and so $\ell_{G,(S,f)}(v)\leq i+1$.
		By the same argument, $\ell_{G,(S,f)}(w)\leq \ell_{G,(S,f)}(v)+1$.
		Note that every layering $(L_0,L_1,\dots)$ of $G$ can be described this way by setting $S=V(G)$ and choosing $f$ so that $f(v)=i$ for each integer $i\geq 0$ and each $v\in L_i$. 
		

		We say that a spanning forest $F$ of $G$ is \defn{$(S,f)$-seeded} if each component of $F$ contains exactly one vertex of $S$, and for each vertex $v\in V(G)\setminus S$ there is exactly one vertex $w$ such that $v\in N_F(w)$ and $\ell_{G,(S,f)}(w)=\ell_{G,(S,f)}(v)-1$. 
		We call $w$ the \defn{parent} of $v$ in $F$ and we call $S$ the set of \defn{root vertices} of $F$.

		\begin{lem}
			\label{GenerateTreeDecomp}
			Let $(S,f)$ be a seed for a plane triangulation $G$, and let $F$ be an $(S,f)$-seeded spanning forest of $G$.
			If $T$ is a spanning tree of $G$ that contains $F$ as a subgraph, and for every edge of $vw$ of $G$ the path $T_{v,w}$ intersects at most $t$ components of $F$, then each bag of the tree-decomposition $\TT$ of $G$ generated by $T$ contains at most $4t$ vertices of each layer of $\mathcal{L}_{G,(S,f)}$.
		\end{lem}
		\begin{proof}
			We will show that for each bag $B_f$ of $\TT$ and each component $T'$ of $F$ we have that $B_f\cap V(T')$ contains at most $2$ vertices of each layer of $\mathcal{L}_{G,(S,f)}$, which will immediately imply the result.
			Recall that $B_f$ is indexed by a face $f$ of $G$, and that $B_f=V(T_{v,w})\cup V(T_{w,z})$, where $v,w$ and $z$ are the three vertices of $G$ incident to $f$. 
			Note that for any path $P$ in $T$ and any pair of vertices $v',w'\in V(P)\cap V(T')$, we have $P_{v',w'}=T_{v',w'}=T'_{v',w'}$.
			Thus $V(P)\cap V(T')$ is a connected subgraph of $P$, and is therefore a path.
			By the definition of $F$, there is a unique vertex $s$ in $T'$ which minimises $\ell_{G,(S,f)}(s)$, and $s\in S$.
			We say that a path $v_0v_1\dots v_\ell$ in $T'$ is \emph{ascending} if $\ell_{G,(S,f)}(v_i)=1+\ell_{G,(S,f)}(v_{i-1})$ for each $i\in [\ell]$.
			It follows inductively from the definition of $F$ that for each $v\in V(T')$ there is an ascending path from $v$ to $s$. 
			Now, given a path $P$ in $T'$ with end-vertices $v'$ and $w'$, $P$ is contained in the union of an ascending path from $v'$ to $s$ and an ascending path from $w'$ to $s$.
			It follows that $V(T_{v,w})\cap V(T')$ contains at most two vertices of each layer, and likewise $V(T_{w,z})\cap V(T')$ contains at most two vertices of each layer.
			Thus, in total $B_f\cap V(T')$ contains at most four vertices of each layer, as required.
		\end{proof}
		
		For a tree $T$ rooted at vertex $r$, a path $P$ in $T$ is  \defn{upward} if $P$ is a subpath of $T_{v,r}$ for some $v\in V(T)$.

		\begin{lem}
			\label{lem:upwardpaths}
			For any integers $n,k,\tau\geq 1$, for any $n$-vertex tree $T$ with root vertex $r$, there is a collection $\mathcal{U}$ of pairwise edge-disjoint upward paths in $T$ such that:
			\begin{enumerate}
				\item\label{item:pathlengths} the length $0$ path $r$ is in $\mathcal{U}$, and every other path in $\mathcal{U}$ has length exactly $k$,
				\item\label{item:depth}    For every vertex $v\in V(T)$, the path $T_{r,v}$ has at most $(3\tau(n/k)^{1/\tau}+1)k +c_vk$ edges, where $c_v$ is the number paths in $\mathcal{U}\setminus \{r\}$ which are properly contained in $T_{r,v}$, and
				\item\label{item:iterations} for every upward path $P$ in $T$ which intersects at least $\tau+1$ paths in $\mathcal{U}$, there is a path $U\in \mathcal{U}$ such $P$ contains the vertex of $U$ which is furthest from $r$.    
			\end{enumerate}    
		\end{lem}
		
		\begin{proof}
			For each node $v\in V(T)$, let $w(v)$ be the number of descendents of $v$. We call $w(v)$ the \defn{weight} of $v$.
			Let $E^*$ be the set of edges $uv$ of $T$ such that $u$ is the parent of $v$ and, for some $i\in [\tau]$, we have that $k(n/k)^{i/\tau}\geq w(u)> k(n/k)^{(i-1)/\tau}$ and $w(v)>w(u)-\frac{1}{2}(n/k)^{(i-1)/\tau}$.
			Note that each component of $(V(T),E^*)$ is an upward path in $T$.
 			Let $\mathcal{C}'$ be a maximum size set of pairwise edge-disjoint paths of length $k$ in the graph $(V(T),E^*)$, such that for any $i\in [k-1]$ and any component $P$ of $(V(T),E^*)$ of length congruent to $i$ mod $k$, the $i$ edges of $P$ which are furthest from $r$ in $T$ are not in $\bigcup \{E(U):U\in \mathcal{C}\}$.
            Let $\mathcal{U}$ be the union of $\mathcal{U}'$ and the length $0$ path $r$.
			
			Clearly $\mathcal{U}$ satisfies \ref{item:pathlengths}. To see that \ref{item:depth} is satisfied, consider an arbitrary vertex $x\in V(T)$, and let $v^*$ be the vertex of $T_{r,x}$ of largest weight such that $w(v^*)\leq k(n/k)^{1/\tau}$.
			Note that $k(n/k)^{\tau/\tau}\geq w(u)$ for every $u\in V(T)$.
			Additionally, for each $uv\in E(P)$ such that $u$ is the parent of $v$, we have that $w(u)\geq w(v)+1$.
			It follows that for each $i\in [\tau]$, there are at most $2(n/k)^{(1/\tau)}$ edges $uv$ in $E(P)\setminus E^*$ such that $k(n/k)^{i/\tau}\geq w(u)> k(n/k)^{(i-1)/\tau}$, where $u$ is the parent of $v$.
			Thus, in total, the subpath of $P$ from $r$ to $v^*$ has at most $2\tau(n/k)^{(1/\tau)}$ edges not in $E^*$. Thus, the restriction of $P_{r,v^*}$ to $E^*$ has at most $1+2\tau(n/k)^{(1/\tau)}$ components. By our choice of $\mathcal{U}'$, each of these components has at most $k-1$ edges which are not in 
			$\bigcup\{E(U):U\in\mathcal{U},U\subseteq P\}$. 
			Thus, if $c_P$ is the number of paths of $\mathcal{U}\setminus \{r\}$ which are properly contained in $P$, then 
			\[|E(P)| \leq (k-1)(1+2\tau(n/k)^{(1/\tau)})+\tau(n/k)^{(1/\tau)}+k(n/k)^{1/\tau}+kc_P<(3\tau(n/k)^{1/\tau}+1)k+kc_P.\]
			
			For each $i\in [\tau]$, let $S_i$ be the set $\{v\in V(T):k(n/k)^{i/\tau}\geq w(v)> k(n/k)^{(i-1)/\tau}\}$, and let $S_0:=V(T)\setminus \bigcup\{S_i:i\in [\tau]\}$.
			Now consider an upward path $P$ which intersects more than $\tau$ paths in $\mathcal{U}$.
			For each $i\in [\tau+1]$ let $U_i$ be a path in $\mathcal{U}$ which intersects $P$ and let $v_i$ be the vertex in $V(P)\cap V(U_i)$ which is furthest from $r$, ordered so that $w(v_i)>w(v_{i+1})$ for all $i\in [\tau]$.
			Note that for each $i\in [\tau]$, $P$ contains a child of $v_i$ which is not in $U_i$.
			We may assume that for each $i\in [\tau]$, $v_i$ is not the furthest vertex of $U_i$ from $r$.
			Now, since for each $i\in [\tau]$ $v_i$ has a child in $U_i$ and all edges of $U$ are in $E^*$, there is a function $f:[\tau]\to [\tau]$ such that $v_i$ is in $S_{f(i)}$ and the child of $v_i$ in $P$ is not in $S_{f(i)}$.
			In particular, $f$ must be the order reversing bijection, so $f(i)=\tau+1-i$ for all $i\in [t]$.
			But now $v_{\tau+1}\in S_0$, which means $v_{\tau+1}$ must be the vertex of $U_{\tau+1}$ which is furthest from $r$.
		\end{proof}
		
		\begin{lem}
			\label{GenerateSeeds}
			For any integer $\tau\geq 1$, any planar triangulation $G$  has a seed $(S,f)$, a spanning tree $T'$, and an $(S,f)$-seeded spanning forest $F\subseteq T'$, such that:
			\begin{enumerate}
				\item $|\mathcal{L}_{G,(S,f)}|\leq (3\tau(|V(G)|/(3\tw(G)+3))^{1/\tau}+2)(3\tw(G)+3)$, and
				\item for each edge $vw$ of $G$, the path $T'_{v,w}$ contains vertices from at most $\tau+2$ components of $F$.
			\end{enumerate}
		\end{lem}
		
		\begin{proof}
			By \cref{lem:twdiamparts}, for an arbitrary vertex $v_0\in V(G)$, there is a tree $T$ rooted at node $r\in V(T)$, and $G$ has a parent-dominated $T$-partition $\PP=(P_x:x\in V(T))$ with weak-diameter at most $3\tw(G)+2$, where $P_r=\{v_0\}$. Apply \cref{lem:upwardpaths} to $T$ with $k:=3\tw(G)+3$ to obtain a set $\mathcal{U}$ of upward paths. For each $U\in \mathcal{U}$, let $r_U$ be the end-vertex of $U$ furthest from $r$ in $T$, let $q_U$ be the end-vertex of $U$ closest to $r$ in $T$, and select an arbitrary vertex $s_U\in P_{r_U}$. Let $S:=\{s_U:U\in \mathcal{U}\}$. For each $U\in \mathcal{U}$, let $\hat{U}$ be the path in $T$ from $r$ to $r_U$.
			Let $\hat{T}$ be the graph obtained from $T$ by identifying $r_U$ with $q_U$ for each $U\in \mathcal{U}$.
			Let $f:S\to \mathbb{Z}$ so that for each $U\in \mathcal{U}$, $f(s_U)=\dist_{\hat{T}}(r,r_U)$.
			Note that for each $U\in \mathcal{U}$, $f(s_U)\leq |E(\hat{U})|-kc_{\hat{U}}$, where $c_{\hat{U}}$ is the number of paths in $\mathcal{U}\setminus \{r\}$ which are properly contained in $\hat{U}$.
			Also note that for each $U\in \mathcal{U}$, $r_U$ separates $V(U)$ from $r$ in $\hat{T}$.

			
			We now verify that $(S,f)$ is a seed.
			Consider distinct paths $U$ and $U'$ in $\mathcal{U}$. 
			Since $U$ and $U'$ are edge-disjoint upward paths, we may assume without loss of generality that $|E(\hat{U})|\leq |E(\hat{U'})|$ and so $\hat{U}$ and $U'$ are edge-disjoint.
			First consider the case where $r_U$ is an ancestor of $r_{U'}$ in $T$. 
			Note that $\dist_{G}(s_{U'},P_{r_U})=\dist_T(r_{U'},r_U)$. Hence, since $\PP$ has weak-diameter at most $k-1$, we have $\dist_T(r_{U'},r_U)\leq \dist_G(s_U,s_{U'})\leq \dist_T(r_{U'},r_U)+k-1$.
			By the definition of $f$, we have $|f(s_{U'})-f(s_U)|\leq\dist_{\hat{T}}(r_U,r_{U'})\leq \dist_T(r_{U'},r_U)-|E(U')|=\dist_T(r_{U'},r_U)-k<\dist_{G}(s_U,s_{U'})$.
			
			Now consider the case where $U$ is not an ancestor of $U'$.
			We aim to show that $f(s_U)\leq f(s_{U'})+\dist_T(r_U,r_{U'})$, which will then imply that $f(s_U)\leq f(s_{U'})+\dist_G(s_U,s_{U'})$.
			Let $P_1$ be the shortest path in $\hat{T}$ from $r$ to $r_U'$, and let $P_2$ be the shortest path in $T$ from $r_U'$ to $U$. 
			Note that $P_2$ corresponds to a walk $W_2$ from $r_U'$ to $r_U$ in $\hat{T}$. Combining $P_1$ and $W_2$, we obtain a walk from $r$ to $r_U$ in $\hat{T}$ of length $f(s_{U'})+\dist_T(r_U,r_{U'})$. Thus $f(s_{U'})+\dist_T(r_U,r_{U'})\leq \dist_{\hat{T}}(r,r_U)=f(s_U)$, as required.
			
			
			Hence, $(S,f)$ is a seed for $G$.
			
			Let $F$ be an $(S,f)$ seeded spanning forest of $G$, and for each $u\in \mathcal{U}$, let $F_U$ be the component of $F$ containing $s_U$. 
			
			\begin{claim}\label{clm:rootparts}
				For all $U\in \mathcal{U}$, we have that $P_{r_U}\subseteq V(F_U)$. 
			\end{claim}
			\begin{proof}
				Suppose for contradiction that there are distinct $U$ and $U'$ in $\mathcal{U}$ such that $F_{U'}$ contains a vertex $v\in P_{r_U}$.
				This implies $\dist_G(v,s_{U'})+f(s_{U'})\leq \dist_G(v,s_{U})+f(s_{U})$.
				Since $\PP$ has weak-diameter at most $k-1$, $\dist_G(v,s_{U})+f(s_{U})\leq \dist_{\hat{T}}(r,r_{U})+k-1$.
				
				Let $P_1$ be a shortest path from $r$ to $r_U'$ in $\hat{T}$, and let $P_2$ be the path in $T$ from $r_U'$ to $r_U$.
				Observe that since $U$ and $U'$ are edge-disjoint upward paths, $P_2$ either contains all of $U$ or all of $U'$, and if $P_2$ contains $r$ then it contains $U\cup U'$. Now $P_2$ corresponds to a walk $W$ in $\hat{T}$ from $r_{U'}$ to $r_U$ in $\hat{T}$.
				Combining this with $P_1$ gives a walk $W'$ of length  $\dist_G(v,s_{U'})+f(s_{U'})$ from $r$ to $r_U$ in $\hat{T}$. Furthermore, $W'$ contains a length $k$ cycle corresponding to either $U$ or $U'$ as a subwalk, so the length of $W'$ is at least $\dist_{\hat{T}}(r,r_U)+k$. Thus $\dist_G(v,s_{U'})+f(s_{U'})>\dist_G(v,s_{U})+f(s_{U})$, a contradiction.

			\end{proof}
				
			
			\begin{claim}\label{clm:quparent}
				If $t\in V(T)$ and $U\in \mathcal{U}$ are such that $t$ is the parent of $q_U$ in $T$, then $P_t\cap V(F_U)=\emptyset$.
			\end{claim}
			\begin{proof}
				Suppose for contradiction that there is a vertex $v\in P_t\cap V(F_U)$.
				Since $F_U$ is connected and contains a vertex in $P_{r_U}$, it also contains a vertex in $P_{q_U}$, and so by Claim~\ref{clm:rootparts} we have that $q_U\notin \{r_{U'}:U'\in \mathcal{U}\setminus \{U\}\}$.
				Let $P$ be a shortest path in $\hat{T}$ from $r$ to $r_U$.
				Note that the set of edges in $T$ which correspond to edges in $E(P)$ in $\hat{T}$ induce a linear forest $L$ in $T$ whose components are paths. 
				Additionally, note that for some $U_0\in \mathcal{U}$, $T_{r_{U_0},q_U}$ is a component of $L$.
				Now $f(s_U)=|E(L)|=\dist_{\hat{T}}(r,r_{U_0})+\dist_T(r_{U_0},q_U)=f(s_{U_0})+\dist_T(r_{U_0},q_U)$.
				Since $\PP$ has weak-diameter at most $k-1$, we have $\dist_G(s_{U_0},v)\leq \dist_T(r_{U_0},t)+k-1\leq \dist_T(r_{U_0},q_U)+k$.
				Finally, we have $\dist_G(s_U,v)\geq \dist_T(r_U,t)=k+1$.
				Combining all of this together, we have $f(s_U)+\dist_G(s_U,v)=f(s_{U_0})+\dist_T(r_{U_0},q_U)+\dist_G(s_U,v)\geq f(s_{U_0})+\dist_T(r_{U_0},q_U)+k+1>f(s_{U_0})+\dist_G(s_{U_0},v)$, which is a contradiction.
			\end{proof}
			
			\begin{claim}\label{clm:central}
				For all $t\in V(T)$ and all $U\in \mathcal{U}$, if $P_t$ contains a vertex of $F_U$, then $T_{t, q_U}$ is an upward path and is disjoint from $\{r_{U'}:U'\in \mathcal{U}\}\setminus \{r_U,q_U\}$.
			\end{claim}
			\begin{proof}
				Let $t\in V(T)$ and $U\in \mathcal{U}$ be such that $P_t$ contains a vertex $v\in V(F_U)$.
				If $T_{t,q_U}$ is not an upward path, then it contains the parent $t'$ of $q_U$ in $T$.
				But then, since $F_U$ is connected, $F_U$ contains a vertex in $P_{t'}$, contradicting Claim~\ref{clm:quparent}.
				
				If $T_{t,q_U}$ contains a vertex $t'$ in $\{r_{U'}:U'\in \mathcal{U}\}\setminus \{r_U,q_U\}$, then $t'$ is also in $T_{t,r_U}$. 
				Again, since $F_U$ is connected, this means $F_U$ contains a vertex in $t'$, which contradicts Claim~\ref{clm:rootparts}.
				
				Thus $T_{t, q_U}$ is an upward path and is disjoint from $\{r_{U'}:U'\in \mathcal{U}\}\setminus \{r_U,q_U\}$, as required.
			\end{proof}
			
			We construct a spanning tree $T'$ from $F$ as follows.
			For each $U\in \mathcal{U}\setminus \{r\}$, let $t_U$ be the closest vertex in $T$ to $r$ such that $V(F_U)\cap P_{t_U}$ is non-empty, and let $v_U\in V(F_U)\cap P_{t_U}$.
			Since $\PP$ is parent-dominated, $v_U$ has a neighbour $v'_U$ in the part of $\PP$ indexed by the parent of $t_U$.
			Pick such a $v'_U$ arbitrarily, and add the edge $v_Uv'_U$ to $T'$.
			
			Now, consider an edge $vv'$ of $G$ which is not in $T'$.
			There are vertices $t$ and $t'$ in $T$ such that $t'$ is the parent of $t$ and $\{v,v'\}\subseteq P_t\cup P_{t'}$.
			Let $r'=t'$ if $t'\in \{r_U:U\in \mathcal{U}\}$ and otherwise let $r'$ be the unique ancestor of $t'$ such that $r'\in \{r_U:U\in \mathcal{U}\}$ and the upward path $T_{t,r'}$ is internally disjoint from $\{r_U:U\in \mathcal{U}\}$.
			Let $\hat{\mathcal{U}}$ be the set of paths $U\in \mathcal{U}$ such that $V(U)\cap V(T_{t,r'})$ and $V(F_U)\cap \bigcup \{P_x:x\in V(T_{t,r'})\}$ are both non-empty.
			By the third property in \cref{lem:upwardpaths}, we have that $|\hat{\mathcal{U}}|\leq \tau+2$.
			
			From Claims~\ref{clm:rootparts} and~\ref{clm:central}, we see that $\bigcup \{P_x:x\in V(T_{t,r'})\}\subseteq \bigcup \{V(F_U):U\in \hat{\mathcal{U}}\}$, and in particular that $P_{r'}$ is contained in a single component of $F$.
			To complete the proof, we will show that $T'[\bigcup \{V(F_U):U\in \hat{\mathcal{U}}]$ is connected, thus demonstrating that $T'_{v,v'}$ intersects at most $\tau+2$ components of $F$.
			
			Let $U_0\in \hat{\mathcal{U}}$ be the path such that $P_{r'}\subseteq V(F_{U_0})$.
			For each $U\in \hat{\mathcal{U}}\setminus U_0$, note that since $T_{t,r'}$ is an upward path and $V(U)\cap V(T_{t,r'})$ and $V(F_U)\cap \bigcup \{P_x:x\in V(T_{t,r'})\}$ are both non-empty, $T_{t,r'}$ contains $t_U$.
			For each $i\in [\dist_T(t,r')]\cup \{0\}$, let $\mathcal{F}_i$ be the set $\{U\in \hat{\mathcal{U}}:\dist_T(t_U,r')=i\}$.
			By construction, for each $i\in [\dist_T(t,r')]$ and each $U\in \mathcal{F}_i$, there is an edge in $T'$ joining $F_U$ to some tree in $\{F_{U'}:U'\in \mathcal{F}_{i-1}\}$.
			Since $\mathcal{F}_0=\{U_0\}$, it follows that $T[\bigcup \{V(F_U):U\in \hat{\mathcal{U}}]$ is connected.
			
			Now consider a vertex $v\in V(G)$, and let $t\in V(T)$ be such that $v\in P_t$.
			Let $Q$ be the upward path $T_{r,t}$, and let $U$ be the path in $\mathcal{U}$ such that $r_U\in V(Q)$ and $T_{r_U,t}$ contains no vertex in $\{r_{U'}:U'\in \mathcal{U}\}\setminus \{r_U\}$.
			Let $c_Q$ be the number of paths in $\mathcal{U}\setminus \{r\}$ which are properly contained in $Q$.
			By the second property of \cref{lem:upwardpaths}, we have $|E(Q)|\leq (3\tau(|V(T)|/k)^{1/\tau}+1)k+kc_Q$.
			By our choice of $U$, there are $c_Q$ paths in $\mathcal{U}\setminus \{r\}$ which are properly contained in $T_{r,r_U}$.
			Now, since $\PP$ has weak-diameter at most $k-1$, we have $f(s_U)+\dist_G(s_U,v)\leq \dist_{\hat{T}}(r,r_U)+(k-1)+\dist_T(r_U,t)\leq \dist_T(r,r_U)-kc_Q+(k-1)+\dist_T(r_U,t)=|E(Q)|-kc_Q+k-1\leq (3\tau(|V(T)|/k)^{1/\tau}+2)k-1$.
			Thus, $v$ is contained in one of the first $(3\tau(|V(T)|/k)^{1/\tau}+2)k$ layers.
			We may assume $|V(T)|\leq |V(G)|$, and so $|\mathcal{L}_{G,(S,f)}|\leq(3\tau(|V(G)|/k)^{1/\tau}+2)k$.
		\end{proof}
		
		We now complete the proof of our main lemma, which which we restate here for convenience.
		\MainLemName*
		
		\begin{proof}
			Let $\tau:=\ceil{\frac{1}{\epsilon}}$ and $k:=\tw(G)+1$. Note that $n\geq k$ and that $\tau\leq \frac{1}{\epsilon}+1$. 
			We may assume that $k> 12+\frac{4}{\epsilon}$, since otherwise an optimal tree-decomposition of $G$ and the layering of $G$ with exactly one non-empty layer satisfy the lemma.  
			By \cref{Triangulate}, there is a planar triangulation $G'$ with $V(G')=V(G)$ and $E(G)\subseteq E(G')$ and $\tw(G')=\tw(G)$. 
			By \cref{GenerateSeeds}, $G'$ has a seed $(S,f)$, 
			a spanning tree $T'$, and 
			an $(S,f)$-seeded spanning forest $F\subseteq T'$ such that:
			\begin{enumerate}
				\item $|\mathcal{L}_{G',(S,f)}|\leq 
				(3\tau(\frac{n}{3k})^{1/\tau}+2)(3k)
				\leq
				9\tau k^{1-\epsilon}n^\epsilon +6k
				\leq
				(\frac{9}{\epsilon}+15) k^{1-\epsilon}n^\epsilon 
				$, and
				\item for every edge $vw$ of $G'$, the path $T'_{v,w}$ contains vertices from at most $\tau+2$ components of $F$.
			\end{enumerate}
			
			Let $\TT$ be the tree-decomposition of $G'$ generated by $T'$. 
			By \cref{GenerateTreeDecomp}, each bag of $\TT$ has at most $4(\tau+2)\leq 12+\frac{4}{\epsilon}$ vertices in each layer of $\mathcal{L}_{G,(S,f)}$. This completes the requirements of \cref{MainLemma}.
		\end{proof}
		
		\section{Extensions via shallow minors}
        \label{sec:Extensions}

        This section shows how to generalise our product structure theorems for planar graphs for more general `beyond planar' graph classes. For any integer $r\geq 0$, a graph $H$ is an \defn{$r$-shallow minor} of a graph $G$ if a graph isomorphic to $H$ can be obtained from $G$ by contracting pairwise-disjoint connected subgraphs of $G$, each with radius at most $r$, and then taking a subgraph. 
        For any integer $s\geq 0$, $H$ is an \defn{$(r,s)$-shallow minor} of $G$ if these connected subgraphs are rooted trees in which each non-root vertex has degree at most $s$ is at distance at most $r$ from the root.
        For a graph $L$ and integer $r\geq 1$, let $L^r$ be the \defn{$r$-power} of $L$, which is the graph with vertex-set $V(L)$, where $vw\in E(L^r)$ if and only if $\dist_L(v,w)\in\{1,\dots,r\}$. The following result of \citet[Theorem~7]{HW24}, which builds on earlier work of \citet{DMW23}, shows that product structure is inherited when taking shallow minors. 
		
		\begin{lem}[\citep{HW24}]
		\label{ShallowMinor}
		For any graphs $H$ and $L$ and integers $r,\ell,t\geq 1$, if a graph $G$ is an $r$-shallow minor of $H \boxtimes L \boxtimes K_{\ell}$ where $H$ has treewidth at most $t$ and $\Delta(L^r)\leq k$, then $G \subsetsim J \boxtimes L^{2r+1} \boxtimes K_{\ell(k+1)}$ for some graph $J$ with treewidth at most $\binom{2r+1+t}{t}-1$.
		\end{lem}
		
		This result can be combined with \cref{PGPST-sqrtn} to bound the length of the path in our product structure theorems.
		
		\begin{lem}
		\label{PlanarShallowMinor}
        For any planar graph $Q$ and integers $n,d,r,\ell\geq 1$, if an $n$-vertex graph $G$ with average degree $d$ is an $r$-shallow minor of $Q\boxtimes K_\ell$, then $G \subsetsim J \boxtimes P \boxtimes K_c$ for some graph $J$ with treewidth at most $\binom{2r+4}{3}-1$, some path $P$ with $|V(P)|\leq 	\ceil{\tfrac{32+52\epsilon}{(2r+1)\epsilon} ((1+rd)n)^{(1+\epsilon)/2}}$, and some integer $c\leq (24+\frac{8}{\epsilon})\ell(2r+1)^2$.
       \end{lem}
		
		\begin{proof}
		We may assume that $Q$ is vertex-minimal such that $G$ is an $r$-shallow minor of $Q\boxtimes K_\ell$. 
        If we select a central representative vertex for each vertex of $G$ in the corresponding connected subgraph of $Q\boxtimes K_{\ell}$, then edges of $G$ naturally correspond to paths in $Q\boxtimes K_{\ell}$ with at most $2r$ internal vertices (where distinct paths are not necessarily internally disjoint).
        By the minimality of $Q$, we have $|V(Q)|\leq |V(G)|+2r|E(G)|= (1+rd)n$. By \cref{PGPST-sqrtn}, $Q\subsetsim H \StrongProd P \StrongProd K_c$ for some planar graph $H$ with $\tw(H)\leq 3$, for some path $P$ with 
        $$|V(P)|\leq 
			(\tfrac{32}{\epsilon}+52) |V(Q)|^{(1+\epsilon)/2} \leq
			(\tfrac{32}{\epsilon}+52) ((1+rd)n)^{(1+\epsilon)/2},$$ and for some integer 
		$c\leq 24+\frac{8}{\epsilon}$. Thus $G$ is an $r$-shallow minor of $H \StrongProd P \StrongProd K_{c\ell}$. So \cref{ShallowMinor} is applicable with $L=P$ and $k=2r$ and $t=3$. So $G \subsetsim J \boxtimes P^{2r+1} \boxtimes K_{c\ell(2r+1)}$ for some graph $J$ with treewidth at most $\binom{2r+4}{3}-1$.   
			Note that $P^{2r+1}\subsetsim P'\boxtimes K_{2r+1}$, where $P'$ is a path with 
			\begin{align*}
				|V(P')| \leq \ceil{ |V(P)|/(2r+1)} 
				& \leq 	\ceil{(\tfrac{32}{\epsilon}+52) ((1+rd)n)^{(1+\epsilon)/2} / (2r+1)}.
			\end{align*}
			Thus $G \subsetsim J \boxtimes P' \boxtimes K_{c\ell(2r+1)^2}$.
		\end{proof}
		
		We now give an example of \cref{PlanarShallowMinor}. \citet{KU22} introduced the following graph class (also see \citep{BG20,CPS22}). A graph $G$ is \defn{fan-planar} if it has a drawing in the plane such that for each edge $e \in  E ( G )$, the edges that cross $e$ have a common end-vertex $v$ and they cross $e$ from the same side (when directed away from $v$). \citet{HW24} proved that every fan-planar graph $G$ is contained in $H \boxtimes P \boxtimes K_{81}$ for some graph $H$ with treewidth at most $19$ and for some path $P$. We show an analogous result with $P$ short. 
		
		\begin{thm}
        \label{FanPlanarProduct}
		For any $\epsilon\in (0,\frac12)$ and integer $n\geq 1$, every $n$-vertex fan-planar graph $G$ is contained in $H \boxtimes P \boxtimes K_{c}$ for some graph $H$ with treewidth at most $19$, some path $P$ with $|V(P)|< 	(\tfrac{11}{\epsilon}+18) (11n)^{(1+\epsilon)/2}$, and some integer $c\leq
        \frac{216}{\epsilon}+648$. 
		\end{thm}
		
		\begin{proof}
		\citet[Lemma~33]{HW24} proved that $G$ is a $1$-shallow minor of $Q \boxtimes K_3$ for some planar graph $Q$. \citet{KU22} showed that $G$ has average degree less than $10n$. So \cref{PlanarShallowMinor} is applicable with $r=1$, $d=10$ and $\ell=3$. So $G \subsetsim J \boxtimes P \boxtimes K_{c}$ for some graph $J$ with treewidth at most $\binom{6}{3}-1=19$, some path $P$ with $|V(P)|\leq 	\ceil{\tfrac{32+52\epsilon}{3\epsilon} (11n)^{(1+\epsilon)/2}}< (\tfrac{11}{\epsilon}+18) (11n)^{(1+\epsilon)/2}$, and some integer 
        $c\leq 27(24+\frac{8}{\epsilon})$.
        \end{proof}

Note that for fan-planar graphs the argument in 
the proof of \cref{PlanarShallowMinor} showing that the graph $Q$ in the proof of \cref{FanPlanarProduct} 
satisfies $|V(Q)|\leq 11n$ is due to Robert Hickingbotham~[personal communication, 2025]. 

\cref{PlanarShallowMinor} can be applied for a variety of beyond planar classes (see \citep{HW24}). For example, a graph $G$ is \defn{$k$-planar} if it has a drawing in the plane, where each edge is crossed at most $k$ times, and no three edges cross at a single point. \cref{PlanarShallowMinor} can be used to show a product structure theorem for $k$-planar graphs with $\tw(J)\in O(k^3)$ and with short paths. We now follow an approach of \citet{DHSW24} that leads to constant bounds on  $\tw(J)$, and still with short paths. The next lemma is a slight extension of Lemma~3.3 in \citep{DHSW24} including bounds on $|V(P)|$. For a graph $G$  and integer $\ell\geq 0$, a partition of $V(G)$ is \defn{$\ell$-blocking} if each part induces a connected subgraph of $G$, and every path of length greater than $\ell$ intersects some part in at least two vertices.

\begin{lem}[\citep{DHSW24}] 
\label{MinorBlockingShallowProduct} 
Let $\GG$ be a minor-closed class such that for some function $f$, non-decreasing function $\alpha$, and integers $\ell,t,c\geq 1$, 
\begin{itemize}
\item every graph \(G\in \GG\) has an $\ell$-blocking partition $\RR$ with width at most $f(\Delta(G))$;
\item  every graph $G\in\GG$ is contained in $H \boxtimes P \boxtimes K_c$  for some graph $H$ with $\tw(H)\leq t$ and for some path $P$ with $|V(P)|\leq \alpha(|V(G)|)$. 
\end{itemize}
Then there is a function $g$ such that for any  integers $r,m\geq 0$ and $d,s,n\geq 1$ and any graph $G\in\GG$, every $(r,s)$-shallow minor $G'$ of $G\boxtimes K_{d}$ with $n$ vertices and $m$ edges is contained in $J \boxtimes P \boxtimes K_{g(d,r,s,\ell,c)}$ for some graph $J$ with $\tw(J)\leq \binom{2\ell+5+t}{t}-1$ and for some path $P$ with  $|V(P)|\leq \ceil{ \alpha((2\ell+4)m + n) / (2\ell+5)}$. 
\end{lem}

\begin{proof} 
By Lemma~3.5 in \citep{DHSW24}, $G'$ is an $(\ell+2)$-shallow minor of  $Q \boxtimes K_{h(d,r,s,\ell)}$ for some minor $Q$ of $G$ and for some function $h$. 
Thus $Q\in \GG$. 
There is a natural map from vertices of $G$ to vertices of $Q\boxtimes K_{h(d,r,s,\ell)}$ and from edges of $G'$ to paths of $Q\boxtimes K_{h(d,r,s,\ell)}$ with at most $2\ell+4$ internal vertices which witnesses that $G'$ is an $(r,s)$-shallow minor of $G\boxtimes K_{d}$. 
We may assume that $Q$ is vertex-minimal, implying $|V(Q)|\leq (2\ell+4)m+n$. 
By assumption, $Q$ is contained in $H \boxtimes P \boxtimes K_c$  for some graph $H$ with $\tw(H)\leq t$ and for some path $P$ with $|V(P)|\leq \alpha(|V(Q)|) \leq \alpha((2\ell+4)m + n)$. Hence $G'$ is an $(\ell+2)$-shallow minor of $H \boxtimes P \boxtimes K_{c\,h(d,r,s,\ell)}$. 
By \cref{ShallowMinor} with $r=\ell+2$ and $k=2\ell+4$, 
$G'$ is contained in $ J \boxtimes P^{2\ell+5} \boxtimes K_{c(2\ell+5)\cdot h(d,r,s,\ell)}$ for some graph $J$ with $\tw(J)\leq \binom{2\ell+5+t}{t}-1$. The result follows with $g(d,r,s,\ell,c) := c(2(\ell+2)+1)^2\cdot h(d,r,s,\ell)$. 
Now $P^{2\ell+5}\subsetsim P' \boxtimes K_{2\ell+5}$ for some path $P'$ with $|V(P')|\leq \ceil{ |V(P)| / (2\ell+5)} \leq \ceil{ \alpha((2\ell+4)m + n) / (2\ell+5)}$. By construction, $G'\subsetsim J \boxtimes P' \boxtimes K_{c(2\ell+5)^2\cdot h(d,r,s,\ell)}$. 
\end{proof}
		

\citet[Lemma~3.1]{DHSW24} showed that every planar graph $G$ has a 222-blocking partition with width at most $f(\Delta(G))$ for some function $f$. \cref{PGPST-sqrtn} implies that for any $\epsilon\in (0,\frac12)$, \cref{MinorBlockingShallowProduct} is applicable for planar graphs with 
$\alpha(n)=(\frac{32}{\epsilon}+52) n^{(1+\epsilon)/2}$ and $\ell=222$ and $t=3$ and $c=24+\frac{8}{\epsilon}$. Thus\footnote{ We use the approximation $\ceil{ (\frac{32}{\epsilon}+52) (448m + n)^{(1+\epsilon)/2} /449 }\leq (\tfrac{7}{\epsilon}+13)(n+m)^{(1+\epsilon)/2}$, which is true for all values of parameters within the specified domains.}:

\begin{lem}
\label{PlanarBlockingShallowProduct}
For any $\epsilon\in (0,\frac12)$, there is a function $g_{\epsilon}$ such that for any  integers $r,m\geq 0$ and $d,s,n\geq 1$ and planar graph $G$, every $(r,s)$-shallow minor of $G\boxtimes K_{d}$ with $n$ vertices and $m$ edges is contained in $J \boxtimes P \boxtimes K_{g_{\epsilon}(d,r,s)}$ for some graph $J$ with $\tw(J)\leq \binom{452}{3}-1$ and some path $P$ with $|V(P)|\leq (\tfrac{7}{\epsilon}+13)(n+m)^{(1+\epsilon)/2}$. 
\end{lem}
		
\begin{thm}
There is a function $g$ such that for every $\epsilon\in(0,\tfrac12)$ and integer $k\geq 0$, every $k$-planar graph with $n$ vertices is contained in $J \boxtimes P \boxtimes K_{g(\epsilon,k)}$ for some graph $J$ with $\tw(J)\leq \binom{452}{3}-1$ and for some path $P$ with  $|V(P)|\leq  (\tfrac{7}{\epsilon}+13)((5\sqrt{k}+1)n)^{(1+\epsilon)/2}$.
\end{thm}
		
\begin{proof}
Let $G$ be a $k$-planar graph with $n$ vertices and $m$ edges.
\citet{PachToth97} showed that $m\leq  4.108\sqrt{k}n$. \citet{HW24} showed that $G$ is a $(\frac{k}{2},4)$-shallow minor of $Q\boxtimes K_2$ for some planar graph $Q$. By \cref{PlanarBlockingShallowProduct}, $G$ is contained in $J \boxtimes P \boxtimes K_{g_{\epsilon}(2,\frac{k}{2},4)}$ for some graph $J$ with $\tw(J)\leq \binom{452}{3}-1$ and for some path $P$ with  $|V(P)|\leq (\tfrac{7}{\epsilon}+13)(n+m)^{(1+\epsilon)/2}\leq (\tfrac{7}{\epsilon}+13)((5\sqrt{k}+1)n)^{(1+\epsilon)/2}$. Thus the result holds with $g(\epsilon,k)=g_{\epsilon}(2,\frac{k}{2},4)$.
\end{proof}

Now we prove a similar product structure theorem for powers of bounded degree planar graphs.

\begin{thm}
There is a function $g$ such that for all $\epsilon\in (0,\tfrac12)$, integers $k,\Delta,n\geq 1$, and every $n$-vertex planar graph $G$ with maximum degree $\Delta$, the $k$-power $G^k$ is contained in $J \boxtimes P \boxtimes K_{g(\epsilon,k,\Delta)}$ for some graph $J$ with $\tw(J)\leq \binom{452}{3}-1$ and some path $P$ with  $|V(P)|\leq  (\frac{7}{\epsilon}+13) ((\tfrac12 \Delta^k+1) n)^{(1+\epsilon)/2}$.
\end{thm}
		
\begin{proof}
Let $d$ be the maximum degree of $G^{\floor{k/2}}$, which is at most $\Delta^{\floor{k/2}}$. 
Let $m:=|E(G^k)|\leq \frac12 \Delta^k n$. \citet[Lemma~25]{HW24} showed that $G^k$ is a $(\floor{\frac{k}{2}},\Delta)$-shallow minor of $G\boxtimes K_{d+1}$, and thus of $G\boxtimes K_{\Delta^{\floor{k/2}}+1}$. By \cref{PlanarBlockingShallowProduct}, $G^k$ is contained in $J \boxtimes P \boxtimes K_{g_{\epsilon}(\Delta^{\floor{k/2}}+1,\floor{\frac{k}{2}},\Delta)}$ for some graph $J$ with $\tw(J)\leq \binom{452}{3}-1$ and some path $P$ with $|V(P)|\leq 
(\tfrac{7}{\epsilon}+13)(n+m)^{(1+\epsilon)/2}
 \leq  (\frac{7}{\epsilon}+13) ((\tfrac12 \Delta^k+1) n)^{(1+\epsilon)/2} 
 $.
\end{proof}

		\section{Open problems}
        \label{sec:openproblems}
		
		We conclude with a number of open problems that arise from this work:
		
		\begin{enumerate}
			\item Can a version of the Planar Graph Product Structure Theorem be proved with $|V(P)|\leq O(\sqrt{n})$ or with $|V(P)|\leq O(\tw(G))$. 
			
			\item It is natural to ask how large the graph $H$ must be in the Planar Graph Product Structure Theorem. Say $G\subsetsim H\boxtimes P \boxtimes K_c$ for some path $P$. Partitioning the vertices of $H\boxtimes P\boxtimes K_c$ according to their second coordinate corresponds to a path-partition of $G$. So $|V(H)|\geq \frac{\ppw(G)}{c}$, which can be linear in $|V(G)|$ (for example, if $G$ contains a vertex whose degree is linear in $|V(G)|$). It would be interesting to know whether path-partition-width is a tight bound for $|V(H)|$. The following is even possible.
			
			\begin{conj}
					There is an integer $c\geq 1$ such that for any integers $n,k,b\geq 1$ with $n\leq kb$, every $n$-vertex planar graph $G$ with treewidth at most $k$ and path-partition-width at most $b$ is contained in $H \boxtimes P \boxtimes K_c$ for some planar graph $H$ with treewidth at most $3$ and $|V(H)|\leq b$ and some path $P$ with $|V(P)|\leq k$.
				\end{conj}

			\item Analogues of the Planar Graph Product Structure hold for graphs embeddable on any fixed surface~\citep{DJMMUW20,DHHW22} and for graphs which exclude a fixed apex graph as a minor~\citep{DJMMUW20,DHHJLMMRW}. Can these results be proved with $|V(P)|\leq O(n^{1-\epsilon})$ for some fixed $\epsilon>0$?
			
		\end{enumerate}
		
		\subsection*{Acknowledgements}
		
		Thanks to Robert Hickingbotham for helpful conversations about fan-planar graphs. 
		
		{
			\fontsize{10pt}{11pt}
			\selectfont
			\bibliographystyle{DavidNatbibStyle}
			\bibliography{DavidBibliography}

\def\soft#1{\leavevmode\setbox0=\hbox{h}\dimen7=\ht0\advance \dimen7 by-1ex\relax\if t#1\relax\rlap{\raise.6\dimen7 \hbox{\kern.3ex\char'47}}#1\relax\else\if T#1\relax \rlap{\raise.5\dimen7\hbox{\kern1.3ex\char'47}}#1\relax \else\if d#1\relax\rlap{\raise.5\dimen7\hbox{\kern.9ex \char'47}}#1\relax\else\if D#1\relax\rlap{\raise.5\dimen7 \hbox{\kern1.4ex\char'47}}#1\relax\else\if l#1\relax \rlap{\raise.5\dimen7\hbox{\kern.4ex\char'47}}#1\relax \else\if L#1\relax\rlap{\raise.5\dimen7\hbox{\kern.7ex \char'47}}#1\relax\else\message{accent \string\soft \space #1 not defined!}#1\relax\fi\fi\fi\fi\fi\fi}
\begin{thebibliography}{60}
\providecommand{\natexlab}[1]{#1}
\providecommand{\msn}[1]{MR:\,\href{http://www.ams.org/mathscinet-getitem?mr=MR{#1}}{#1}}
\providecommand{\ZBL}[1]{Zbl:\,\href{https://www.zentralblatt-math.org/zmath/en/search/?q=an:#1}{#1}}
\providecommand{\url}[1]{\texttt{#1}}
\providecommand{\urlprefix}{}
\expandafter\ifx\csname urlstyle\endcsname\relax
  \providecommand{\doi}[1]{doi:\discretionary{}{}{}#1}\else
  \providecommand{\doi}{doi:\discretionary{}{}{}\begingroup \urlstyle{rm}\Url}\fi

\bibitem[{Arnborg and Proskurowski(1986)}]{AP-SJADM96}
\textsc{Stefan Arnborg and Andrzej Proskurowski}.
\newblock \href{https://doi.org/10.1137/0607033}{Characterization and recognition of partial $3$-trees}.
\newblock \emph{SIAM J. Algebraic Discrete Methods}, 7(2):305--314, 1986.

\bibitem[{Bekos and Grilli(2020)}]{BG20}
\textsc{Michael~A. Bekos and Luca Grilli}.
\newblock \href{https://doi.org/10.1007/978-981-15-6533-5\_8}{Fan-planar graphs}.
\newblock In \textsc{Seok{-}Hee Hong and Takeshi Tokuyama}, eds., \emph{Beyond Planar Graphs}, pp. 131--148. Springer, 2020.

\bibitem[{Bekos et~al.(2024)Bekos, Lozzo, Hlinen{\'{y}}, and Kaufmann}]{BDHK24}
\textsc{Michael~A. Bekos, Giordano~Da Lozzo, Petr Hlinen{\'{y}}, and Michael Kaufmann}.
\newblock \href{https://doi.org/10.37236/12123}{Graph product structure for $h$-framed graphs}.
\newblock \emph{Electron. J. Comb.}, 31(4), 2024.

\bibitem[{Biedl and Vel{\'{a}}zquez(2013)}]{BV13}
\textsc{Therese Biedl and Lesvia Elena~Ruiz Vel{\'{a}}zquez}.
\newblock \href{https://doi.org/10.1016/j.comgeo.2012.09.004}{Drawing planar 3-trees with given face areas}.
\newblock \emph{Comput. Geom.}, 46(3):276--285, 2013.

\bibitem[{Bodlaender(1998)}]{Bodlaender98}
\textsc{Hans~L. Bodlaender}.
\newblock \href{https://doi.org/10.1016/S0304-3975(97)00228-4}{A partial $k$-arboretum of graphs with bounded treewidth}.
\newblock \emph{Theoret. Comput. Sci.}, 209(1-2):1--45, 1998.

\bibitem[{Bodlaender(1999)}]{Bodlaender-DMTCS99}
\textsc{Hans~L. Bodlaender}.
\newblock \href{https://dmtcs.episciences.org/256}{A note on domino treewidth}.
\newblock \emph{Discrete Math. Theor. Comput. Sci.}, 3(4):141--150, 1999.

\bibitem[{Bodlaender and Engelfriet(1997)}]{BodEng-JAlg97}
\textsc{Hans~L. Bodlaender and Joost Engelfriet}.
\newblock \href{https://doi.org/10.1006/jagm.1996.0854}{Domino treewidth}.
\newblock \emph{J. Algorithms}, 24(1):94--123, 1997.

\bibitem[{Bodlaender et~al.(2022)Bodlaender, Groenland, and Jacob}]{BGJ22}
\textsc{Hans~L. Bodlaender, Carla Groenland, and Hugo Jacob}.
\newblock \href{https://doi.org/10.4230/LIPIcs.IPEC.2022.7}{On the parameterized complexity of computing tree-partitions}.
\newblock In \textsc{Holger Dell and Jesper Nederlof}, eds., \emph{Proc. 17th International Symposium on Parameterized and Exact Computation \textup{(IPEC 2022)}}, vol. 249 of \emph{LIPIcs}, pp. 7:1--7:20. Schloss Dagstuhl, 2022.

\bibitem[{Bonamy et~al.(2022)Bonamy, Gavoille, and Pilipczuk}]{BGP22}
\textsc{Marthe Bonamy, Cyril Gavoille, and Micha{\l} Pilipczuk}.
\newblock \href{https://doi.org/10.1137/20M1330464}{Shorter labeling schemes for planar graphs}.
\newblock \emph{SIAM J. Discrete Math.}, 36(3):2082--2099, 2022.

\bibitem[{B\"{o}ttcher et~al.(2010)B\"{o}ttcher, Pruessmann, Taraz, and W\"{u}rfl}]{BPTW10}
\textsc{Julia B\"{o}ttcher, Klaas~P. Pruessmann, Anusch Taraz, and Andreas W\"{u}rfl}.
\newblock \href{https://doi.org/10.1016/j.ejc.2009.10.010}{Bandwidth, expansion, treewidth, separators and universality for bounded-degree graphs}.
\newblock \emph{European J. Combin.}, 31(5):1217--1227, 2010.

\bibitem[{Campbell et~al.(2024)Campbell, Clinch, Distel, Gollin, Hendrey, Hickingbotham, Huynh, Illingworth, Tamitegama, Tan, and Wood}]{UTW}
\textsc{Rutger Campbell, Katie Clinch, Marc Distel, J.~Pascal Gollin, Kevin Hendrey, Robert Hickingbotham, Tony Huynh, Freddie Illingworth, Youri Tamitegama, Jane Tan, and David~R. Wood}.
\newblock \href{https://doi.org/10.1017/S0963548323000457}{Product structure of graph classes with bounded treewidth}.
\newblock \emph{Combin. Probab. Comput.}, 33(3):351--376, 2024.

\bibitem[{Chen(2016)}]{Chen16}
\textsc{Hao Chen}.
\newblock \href{https://doi.org/10.1007/s00454-016-9777-3}{Apollonian ball packings and stacked polytopes}.
\newblock \emph{Discrete Comput. Geom.}, 55(4):801--826, 2016.

\bibitem[{Cheong et~al.(2022)Cheong, Pfister, and Schlipf}]{CPS22}
\textsc{Otfried Cheong, Maximilian Pfister, and Lena Schlipf}.
\newblock \href{https://doi.org/10.1007/978-3-031-22203-0\_18}{The thickness of fan-planar graphs is at most three}.
\newblock In \textsc{Patrizio Angelini and Reinhard von Hanxleden}, eds., \emph{Proc. 30th International Symposium on Graph Drawing and Network Visualization \textup{({GD} 2022)}}, vol. 13764 of \emph{Lecture Notes in Computer Science}, pp. 247--260. Springer, 2022.

\bibitem[{Ding and Oporowski(1995)}]{DO95}
\textsc{Guoli Ding and Bogdan Oporowski}.
\newblock \href{https://doi.org/10.1002/jgt.3190200412}{Some results on tree decomposition of graphs}.
\newblock \emph{J. Graph Theory}, 20(4):481--499, 1995.

\bibitem[{Ding and Oporowski(1996)}]{DO96}
\textsc{Guoli Ding and Bogdan Oporowski}.
\newblock \href{https://doi.org/10.1016/0012-365X(94)00337-I}{On tree-partitions of graphs}.
\newblock \emph{Discrete Math.}, 149(1--3):45--58, 1996.

\bibitem[{Distel et~al.(2022)Distel, Hickingbotham, Huynh, and Wood}]{DHHW22}
\textsc{Marc Distel, Robert Hickingbotham, Tony Huynh, and David~R. Wood}.
\newblock \href{https://doi.org/10.48550/arXiv.2112.10025}{Improved product structure for graphs on surfaces}.
\newblock \emph{Discrete Math. Theor. Comput. Sci.}, 24(2):\#6, 2022.

\bibitem[{Distel et~al.(2024)Distel, Hickingbotham, Seweryn, and Wood}]{DHSW24}
\textsc{Marc Distel, Robert Hickingbotham, Micha{\l}~T. Seweryn, and David~R. Wood}.
\newblock \href{https://doi.org/10.5802/igt.4}{Powers of planar graphs, product structure, and blocking partitions}.
\newblock \emph{Innovations in Graph Theory}, 1:39--86, 2024.

\bibitem[{Distel and Wood(2024)}]{DW24}
\textsc{Marc Distel and David~R. Wood}.
\newblock \href{https://doi.org/10.1007/978-3-031-47417-0_11}{Tree-partitions with bounded degree trees}.
\newblock In \textsc{David~R. Wood, Jan de~Gier, and Cheryl~E. Praeger}, eds., \emph{2021--2022 MATRIX Annals}, pp. 203--212. Springer, 2024.

\bibitem[{D\k{e}bski et~al.(2021)D\k{e}bski, Felsner, Micek, and Schr\"{o}der}]{DFMS21}
\textsc{Micha{\l} D\k{e}bski, Stefan Felsner, Piotr Micek, and Felix Schr\"{o}der}.
\newblock \href{https://doi.org/10.19086/aic.27351}{Improved bounds for centered colorings}.
\newblock \emph{Adv. Comb.}, \#8, 2021.

\bibitem[{Draganić et~al.(2023)Draganić, Kaufmann, Correia, Petrova, and Steiner}]{DKCPS}
\textsc{Nemanja Draganić, Marc Kaufmann, David~Munhá Correia, Kalina Petrova, and Raphael Steiner}.
\newblock \href{http://arxiv.org/abs/2307.12028}{Size-{R}amsey numbers of structurally sparse graphs}.
\newblock 2023, arXiv:2307.12028.

\bibitem[{Dujmovi\'c et~al.(2021)Dujmovi\'c, Esperet, Gavoille, Joret, Micek, and Morin}]{DEGJMM21}
\textsc{Vida Dujmovi\'c, Louis Esperet, Cyril Gavoille, Gwena\"el Joret, Piotr Micek, and Pat Morin}.
\newblock \href{https://doi.org/10.1145/3477542}{Adjacency labelling for planar graphs (and beyond)}.
\newblock \emph{J. ACM}, 68(6):\#42, 2021.

\bibitem[{Dujmovi{\'c} et~al.(2020{\natexlab{a}})Dujmovi{\'c}, Esperet, Joret, Walczak, and Wood}]{DEJWW20}
\textsc{Vida Dujmovi{\'c}, Louis Esperet, Gwena\"{e}l Joret, Bartosz Walczak, and David~R. Wood}.
\newblock \href{https://doi.org/10.19086/aic.12100}{Planar graphs have bounded nonrepetitive chromatic number}.
\newblock \emph{Adv. Comb.}, \#5, 2020{\natexlab{a}}.

\bibitem[{Dujmovi\'c et~al.(2023)Dujmovi\'c, Hickingbotham, Hodor, Joret, La, Micek, Morin, Rambaud, and Wood}]{DHHJLMMRW}
\textsc{Vida Dujmovi\'c, Robert Hickingbotham, Jędrzej Hodor, Gwena\"el Joret, Hoang La, Piotr Micek, Pat Morin, Clément Rambaud, and David~R. Wood}.
\newblock \href{https://doi.org/10.1137/1.9781611977912.48}{The grid-minor theorem revisited}.
\newblock In \emph{Proc. 2024 Annual ACM-SIAM Symposium on Discrete Algorithms \textup{(SODA '24)}}, pp. 1241--1245. 2023.
\newblock arXiv:2307.02816.

\bibitem[{Dujmovi{\'c} et~al.(2020{\natexlab{b}})Dujmovi{\'c}, Joret, Micek, Morin, Ueckerdt, and Wood}]{DJMMUW20}
\textsc{Vida Dujmovi{\'c}, Gwena\"{e}l Joret, Piotr Micek, Pat Morin, Torsten Ueckerdt, and David~R. Wood}.
\newblock \href{https://doi.org/10.1145/3385731}{Planar graphs have bounded queue-number}.
\newblock \emph{J. ACM}, 67(4):\#22, 2020{\natexlab{b}}.

\bibitem[{Dujmovi\'c et~al.(2025)Dujmovi\'c, Joret, Micek, Morin, and Wood}]{DJMMW25}
\textsc{Vida Dujmovi\'c, Gwena\"el Joret, Piotr Micek, Pat Morin, and David~R. Wood}.
\newblock Planar graphs in blowups of fans.
\newblock In \emph{Proc. Annual ACM-SIAM Symp. on Discrete Algorithms \textup{(SODA '25)}}, pp. 3382--3391. 2025.

\bibitem[{Dujmovi{\'c} et~al.(2017)Dujmovi{\'c}, Morin, and Wood}]{DMW17}
\textsc{Vida Dujmovi{\'c}, Pat Morin, and David~R. Wood}.
\newblock \href{https://doi.org/10.1016/j.jctb.2017.05.006}{Layered separators in minor-closed graph classes with applications}.
\newblock \emph{J. Combin. Theory Ser. B}, 127:111--147, 2017.
\newblock arXiv:1306.1595.

\bibitem[{Dujmovi{\'c} et~al.(2023)Dujmovi{\'c}, Morin, and Wood}]{DMW23}
\textsc{Vida Dujmovi{\'c}, Pat Morin, and David~R. Wood}.
\newblock \href{https://doi.org/10.1016/j.jctb.2023.03.004}{Graph product structure for non-minor-closed classes}.
\newblock \emph{J. Combin. Theory Ser. B}, 162:34--67, 2023.

\bibitem[{Dujmovic et~al.(2024)Dujmovic, Morin, Wood, and Worley}]{DMWW}
\textsc{Vida Dujmovic, Pat Morin, David~R. Wood, and David Worley}.
\newblock \href{https://doi.org/10.48550/arXiv.2402.14181}{Grid minors and products}.
\newblock 2024, arXiv:2402.14181.

\bibitem[{Dvor{\'{a}}k et~al.(2022)Dvor{\'{a}}k, Gon{\c{c}}alves, Lahiri, Tan, and Ueckerdt}]{DGLTU22}
\textsc{Zdenek Dvor{\'{a}}k, Daniel Gon{\c{c}}alves, Abhiruk Lahiri, Jane Tan, and Torsten Ueckerdt}.
\newblock \href{https://doi.org/10.4230/LIPIcs.SoCG.2022.38}{On comparable box dimension}.
\newblock In \textsc{Xavier Goaoc and Michael Kerber}, eds., \emph{Proc. 38th Int'l Symp. on Computat. Geometry \textup{(SoCG 2022)}}, vol. 224 of \emph{LIPIcs}, pp. 38:1--38:14. Schloss Dagstuhl, 2022.

\bibitem[{Dvo{\v{r}}{\'a}k and Sereni(2020)}]{DS20}
\textsc{Zden{\v{e}}k Dvo{\v{r}}{\'a}k and Jean-S\'ebastien Sereni}.
\newblock \href{https://doi.org/10.37236/8909}{On fractional fragility rates of graph classes}.
\newblock \emph{Electronic J. Combinatorics}, 27:P4.9, 2020.

\bibitem[{Edenbrandt(1986)}]{Edenbrandt86}
\textsc{Anders Edenbrandt}.
\newblock \href{https://doi.org/10.1007/BF01933740}{Quotient tree partitioning of undirected graphs}.
\newblock \emph{BIT}, 26(2):148--155, 1986.

\bibitem[{Eppstein(1999)}]{Eppstein99}
\textsc{David Eppstein}.
\newblock \href{https://doi.org/10.7155/jgaa.00014}{Subgraph isomorphism in planar graphs and related problems}.
\newblock \emph{J. Graph Algorithms Appl.}, 3(3):1--27, 1999.

\bibitem[{Esperet et~al.(2023)Esperet, Joret, and Morin}]{EJM23}
\textsc{Louis Esperet, Gwena\"{e}l Joret, and Pat Morin}.
\newblock \href{https://doi.org/10.1112/jlms.12781}{Sparse universal graphs for planarity}.
\newblock \emph{J. London Math. Soc.}, 108(4):1333--1357, 2023.

\bibitem[{Frieze and Tsourakakis(2014)}]{FT14}
\textsc{Alan Frieze and Charalampos~E. Tsourakakis}.
\newblock \href{https://doi.org/10.1080/15427951.2013.796300}{Some properties of random {A}pollonian networks}.
\newblock \emph{Internet Math.}, 10(1-2):162--187, 2014.

\bibitem[{Gajarský et~al.(2025)Gajarský, Pilipczuk, and Pokrývka}]{GPP}
\textsc{Jakub Gajarský, Michał Pilipczuk, and Filip Pokrývka}.
\newblock \href{https://arxiv.org/abs/2501.07558}{{3D}-grids are not transducible from planar graphs}.
\newblock 2025, arXiv:2501.07558.

\bibitem[{Gawrychowski and Janczewski(2022)}]{GJ22}
\textsc{Pawel Gawrychowski and Wojciech Janczewski}.
\newblock \href{https://doi.org/10.1137/1.9781611977066.3}{Simpler adjacency labeling for planar graphs with {B}-trees}.
\newblock In \textsc{Karl Bringmann and Timothy~M. Chan}, eds., \emph{Proc. 5th Symposium on Simplicity in Algorithms \textup{(SOSA@SODA 2022)}}, pp. 24--36. {SIAM}, 2022.

\bibitem[{Halin(1991)}]{Halin91}
\textsc{Rudolf Halin}.
\newblock \href{https://doi.org/10.1016/0012-365X(91)90436-6}{Tree-partitions of infinite graphs}.
\newblock \emph{Discrete Math.}, 97:203--217, 1991.

\bibitem[{Harvey and Wood(2017)}]{HW17}
\textsc{Daniel~J. Harvey and David~R. Wood}.
\newblock \href{https://doi.org/10.1002/jgt.22030}{Parameters tied to treewidth}.
\newblock \emph{J. Graph Theory}, 84(4):364--385, 2017.

\bibitem[{Hickingbotham and Wood(2024)}]{HW24}
\textsc{Robert Hickingbotham and David~R. Wood}.
\newblock \href{https://doi.org/10.1137/22M1540296}{Shallow minors, graph products and beyond-planar graphs}.
\newblock \emph{SIAM J. Discrete Math.}, 38(1):1057--1089, 2024.

\bibitem[{Hliněný and Jedelský(2024)}]{HJ24}
\textsc{Petr Hliněný and Jan Jedelský}.
\newblock \href{https://drops.dagstuhl.de/entities/document/10.4230/LIPIcs.MFCS.2024.61}{$\mathcal{H}$-clique-width and a hereditary analogue of product structure}.
\newblock In \textsc{Rastislav Kr\'{a}lovi\v{c} and Anton{\'\i}n Ku\v{c}era}, eds., \emph{Proc. 49th Int'l Symp. on Math. Foundations of Comput. Sci. \textup{(MFCS 2024)}}, vol. 306 of \emph{LIPIcs}, pp. 61:1--61:16. Schloss Dagstuhl, 2024.

\bibitem[{Hliněný and Jedelský(2025)}]{HJ25}
\textsc{Petr Hliněný and Jan Jedelský}.
\newblock \href{https://arxiv.org/abs/2501.18326}{Transductions of graph classes admitting product structure}.
\newblock 2025, arXiv:2501.18326.

\bibitem[{Huynh et~al.(2021)Huynh, Mohar, {\v{S}}{\'a}mal, Thomassen, and Wood}]{HMSTW}
\textsc{Tony Huynh, Bojan Mohar, Robert {\v{S}}{\'a}mal, Carsten Thomassen, and David~R. Wood}.
\newblock \href{https://arxiv.org/abs/2109.00327}{Universality in minor-closed graph classes}.
\newblock 2021, arXiv:2109.00327.

\bibitem[{Illingworth et~al.(2022)Illingworth, Scott, and Wood}]{ISW-arXiv}
\textsc{Freddie Illingworth, Alex Scott, and David~R. Wood}.
\newblock \href{https://arxiv.org/abs/2104.06627}{Product structure of graphs with an excluded minor}.
\newblock 2022, arXiv:2104.06627.

\bibitem[{Illingworth et~al.(2024)Illingworth, Scott, and Wood}]{ISW24}
\textsc{Freddie Illingworth, Alex Scott, and David~R. Wood}.
\newblock \href{https://doi.org/10.1090/btran/192}{Product structure of graphs with an excluded minor}.
\newblock \emph{Trans. Amer. Math. Soc. Ser. B}, 11:1233--1248, 2024.

\bibitem[{Jacob and Pilipczuk(2022)}]{JP22}
\textsc{Hugo Jacob and Marcin Pilipczuk}.
\newblock \href{https://doi.org/10.1007/978-3-031-15914-5\_21}{Bounding twin-width for bounded-treewidth graphs, planar graphs, and bipartite graphs}.
\newblock In \textsc{Michael~A. Bekos and Michael Kaufmann}, eds., \emph{Proc. 48th Int'l Workshop on Graph-Theoretic Concepts in Comput. Sci. \textup{({WG} 2022)}}, vol. 13453 of \emph{Lecture Notes in Comput. Sci.}, pp. 287--299. Springer, 2022.

\bibitem[{Kaufmann and Ueckerdt(2022)}]{KU22}
\textsc{Michael Kaufmann and Torsten Ueckerdt}.
\newblock \href{https://doi.org/10.37236/10521}{The density of fan-planar graphs}.
\newblock \emph{Electron. J. Combin.}, 29(1), 2022.

\bibitem[{Knauer and Ueckerdt(2012)}]{KU12}
\textsc{Kolja Knauer and Torsten Ueckerdt}.
\newblock \href{https://kam.mff.cuni.cz/workshops/mcw/work18/mcw2012booklet.pdf}{Simple treewidth}.
\newblock In \textsc{Pavel Ryt\'ir}, ed., \emph{Midsummer Combinatorial Workshop Prague}. 2012.

\bibitem[{Kratochv{\'{\i}}l and Vaner(2012)}]{KV12}
\textsc{Jan Kratochv{\'{\i}}l and Michal Vaner}.
\newblock \href{http://arxiv.org/abs/1210.8113}{A note on planar partial 3-trees}.
\newblock arXiv:1210.8113, 2012.

\bibitem[{Kr\'{a}\v{l} et~al.(2024)Kr\'{a}\v{l}, Pek\'{a}rkov\'{a}, and \v{S}torgel}]{KPS24}
\textsc{Daniel Kr\'{a}\v{l}, Krist\'{y}na Pek\'{a}rkov\'{a}, and Kenny \v{S}torgel}.
\newblock \href{https://drops.dagstuhl.de/entities/document/10.4230/LIPIcs.MFCS.2024.66}{Twin-width of graphs on surfaces}.
\newblock In \textsc{Rastislav Kr\'{a}lovi\v{c} and Anton{\'\i}n Ku\v{c}era}, eds., \emph{Proc. 49th Int'l Symp. on Math'l Foundations of Comput. Sci. \textup{(MFCS 2024)}}, vol. 306 of \emph{LIPIcs}, pp. 66:1--66:15. Schloss Dagstuhl, 2024.

\bibitem[{Markenzon et~al.(2006)Markenzon, Justel, and Paciornik}]{MJP06}
\textsc{Lilian Markenzon, Claudia~Marcela Justel, and N.~Paciornik}.
\newblock \href{https://doi.org/10.1016/j.dam.2005.05.021}{Subclasses of {$k$}-trees: characterization and recognition}.
\newblock \emph{Discrete Appl. Math.}, 154(5):818--825, 2006.

\bibitem[{Pach and T{\'o}th(1997)}]{PachToth97}
\textsc{J{\'a}nos Pach and G{\'e}za T{\'o}th}.
\newblock \href{https://doi.org/10.1007/BF01215922}{Graphs drawn with few crossings per edge}.
\newblock \emph{Combinatorica}, 17(3):427--439, 1997.

\bibitem[{Reed(1997)}]{Reed97}
\textsc{Bruce~A. Reed}.
\newblock \href{https://doi.org/10.1017/CBO9780511662119.006}{Tree width and tangles: a new connectivity measure and some applications}.
\newblock In \textsc{R.~A. Bailey}, ed., \emph{Surveys in Combinatorics}, vol. 241 of \emph{London Math. Soc. Lecture Note Ser.}, pp. 87--162. Cambridge Univ. Press, 1997.

\bibitem[{Seese(1985)}]{Seese85}
\textsc{Detlef Seese}.
\newblock \href{https://doi.org/10.1007/BFb0028825}{Tree-partite graphs and the complexity of algorithms}.
\newblock In \textsc{Lothar Budach}, ed., \emph{Proc. Int'l Conf. on Fundamentals of Computation Theory}, vol. 199 of \emph{Lecture Notes Comput. Sci.}, pp. 412--421. Springer, 1985.

\bibitem[{Seymour and Thomas(1993)}]{ST93}
\textsc{Paul Seymour and Robin Thomas}.
\newblock \href{https://doi.org/10.1006/jctb.1993.1027}{Graph searching and a min-max theorem for tree-width}.
\newblock \emph{J. Combin. Theory Ser. B}, 58(1):22--33, 1993.

\bibitem[{Shahrokhi(2013)}]{Shahrokhi13}
\textsc{Farhad Shahrokhi}.
\newblock \href{http://arxiv.org/abs/1502.06175}{New representation results for planar graphs}.
\newblock In \emph{29th European Workshop on Computational Geometry \textup{(EuroCG 2013)}}, pp. 177--180. 2013.
\newblock arXiv:1502.06175.

\bibitem[{Ueckerdt et~al.(2022)Ueckerdt, Wood, and Yi}]{UWY22}
\textsc{Torsten Ueckerdt, David~R. Wood, and Wendy Yi}.
\newblock \href{https://doi.org/10.37236/10614}{An improved planar graph product structure theorem}.
\newblock \emph{Electron. J. Combin.}, 29:P2.51, 2022.

\bibitem[{von Staudt(1847)}]{vonStaudt}
\textsc{Karl Georg~Christian von Staudt}.
\newblock Geometrie der lage.
\newblock Verlag von Bauer and Rapse 25. Julius Merz, N{\"u}rnberg, 1847.

\bibitem[{Wood(2006)}]{Wood06}
\textsc{David~R. Wood}.
\newblock \href{https://doi.org/10.1002/jgt.20183}{Vertex partitions of chordal graphs}.
\newblock \emph{J. Graph Theory}, 53(2):167--172, 2006.

\bibitem[{Wood(2009)}]{Wood09}
\textsc{David~R. Wood}.
\newblock \href{https://doi.org/10.1016/j.ejc.2008.11.010}{On tree-partition-width}.
\newblock \emph{European J. Combin.}, 30(5):1245--1253, 2009.

\bibitem[{Wulf(2016)}]{Wulf16}
\textsc{Lasse Wulf}.
\newblock \href{https://i11www.iti.kit.edu/_media/teaching/theses/ba-wulf-16.pdf}{Stacked treewidth and the {C}olin de {V}erdi\'ere number}.
\newblock 2016.
\newblock Bachelorthesis, Institute of Theoretical Computer Science, Karlsruhe Institute of Technology.

\end{thebibliography}
		}
	\end{document}